\documentclass[letterpaper,11pt,reqno]{amsart} 
\usepackage[margin=1in]{geometry} 
\usepackage{mathrsfs,xfrac} 
\usepackage[colorlinks=true,linkcolor=blue,citecolor=blue,urlcolor=blue]{hyperref} 
\usepackage{amsmath,amssymb,amsthm,amsfonts,amsbsy,latexsym,dsfont, color,graphicx}
\usepackage{xcolor}
\usepackage{subcaption}
\usepackage[foot]{amsaddr}
\usepackage{caption}
\usepackage{regexpatch}
\usepackage{comment}
\usepackage{todonotes}
\usepackage{algorithm}
\usepackage{algpseudocode}
\usepackage[shortlabels]{enumitem}
\newenvironment{enumeratei}{\begin{enumerate}[\upshape i)]}{\end{enumerate}}

\makeatletter
\xpatchcmd{\@todo}{\setkeys{todonotes}{#1}}{\setkeys{todonotes}{inline,#1}}{}{}
\makeatother
\tikzset{
		big dot/.style={
			circle, inner sep=0pt, 
			minimum size=1.2mm, fill=black
		}
	}

\newtheorem{thm}{Theorem}[section]
\newtheorem{lem}[thm]{Lemma}
\newtheorem{cor}[thm]{Corollary}

\newtheorem{conj}[thm]{Conjecture}

\newtheorem{oques}[thm]{Open question}

\theoremstyle{definition}

\newtheorem{rem}[thm]{Remark}

\newtheorem{ex}[thm]{Example}
\renewcommand{\le}{\leqslant}  
\renewcommand{\ge}{\geqslant}

\newcommand{\wt}{\widetilde}

\newcommand{\ind}{\mathds{1}}

\newcommand{\norm}[1]{\left\Vert#1\right\Vert}
\newcommand{\abs}[1]{\left\vert#1\right\vert}

\newcommand{\ie}{{i.e.,}}

\newcommand{\eg}{{e.g.,}}
 
\newcommand{\equald}{\stackrel{\mathrm{d}}{=}}

\def\qed{ \hfill $\blacksquare$}  
   \let\gd=\delta 
           \let\go=\omega   \let\gs=\sigma  
  
\let\gC=\Gamma \let\gD=\Delta

\newcommand{\cE}{\mathcal{E}}\newcommand{\cF}{\mathcal{F}}
\newcommand{\cH}{\mathcal{H}}

\newcommand{\cM}{\mathcal{M}}

\newcommand{\cT}{\mathcal{T}}

\newcommand{\mvx}{\boldsymbol{x}}

\newcommand{\mvgd}{\boldsymbol{\delta}}

\newcommand{\mvpsi}{\boldsymbol{\psi}}   

\newcommand{\bR}{\mathbb{R}}

\newcommand{\dN}{\mathds{N}}

\newcommand{\dR}{\mathds{R}}

\DeclareMathOperator{\E}{\mathds{E}}
\DeclareMathOperator{\pr}{\mathds{P}}

\DeclareMathOperator{\var}{Var}
\DeclareMathOperator{\cov}{Cov}

\providecommand{\dwas}{\textrm{d}_\mathcal{W}} 

\newcommand{\pam}{\ensuremath{\mathrm{PAM}_n}}
\DeclareMathOperator{\N}{N}
\DeclareMathOperator{\poi}{Poisson}
\DeclareMathOperator{\dtv}{d_{TV}}

\begin{document}
\title[Triangle counts in LPAM]{Limiting distributions of triangle counts in linear preferential attachment models}
\author[Dey]{Partha S.~Dey}
\author[Terlov]{Grigory Terlov}

\begin{abstract}
We derive distributional approximations for the number of triangles in the linear preferential attachment model $\pam(m,\delta)$, where $m\ge 2$ and $\delta>-m$, with explicit rates of convergence. The limiting distribution undergoes a phase transition from Gaussian to another nontrivial distribution, which we characterize explicitly. The asymptotic behavior is governed by the interplay between the hidden random environment and the mean-field interaction effect.
In particular, our analysis also yields a continuous phase transition in the expected number of triangles as $\delta$ varies.
\end{abstract}
\date{\today}
\maketitle
\setcounter{tocdepth}{1}\tableofcontents

\section{Introduction and main results}

In the study of random graphs and complex networks, the number of triangles plays a central role in quantifying local clustering and higher-order connectivity patterns. It provides insight into the structural properties and is closely tied to community formation and transitivity in real-world systems. For instance, the standard measure of clustering in a network, the clustering coefficient, is defined as the ratio of the number of triangles to the number of wedges (\ie\ paths of length 2).

Preferential attachment models (PAM), introduced by Barab\'asi and Albert in 1999~\cite{BA99}, are among the most studied classes of evolving random graphs. Their appeal lies in a simple “rich-get-richer’’ growth mechanism, where new vertices are more likely to attach to vertices of large degree. This generates scale-free degree sequences that typically exhibit power-law tails, consistent with observations across a range of networks, including citation graphs, the World Wide Web, and various social networks. At the same time, PAMs remain sufficiently tractable from the perspective of rigorous analysis, especially when the attachment function is linear.

In this paper, we study a variant of the linear preferential attachment model $\pam(m, \delta)$, where $n\in \dN$ denotes the number of vertices. 
We refer the reader to~\cite[Chapter~5]{RemcoBook2} for a systematic overview of the most common variants. 
Since different variants of the model typically yield comparable large-scale (mean) behavior, the choice of the variant primarily affects which technical tools are available. On the other hand, it is not a priori clear that the same is true for distributional approximations. The variant considered in the present work is constructed recursively as follows:
\begin{algorithmic}
 \State \textbf{Initial data:} $m\in \dN$ and $\delta>-m$.
 
 \State $\mathbf{n=2}$:~$\mathrm{PAM}_2(m, \delta)$ consists of two vertices $\{1,2\}$ with $m$ edges between them.
 
 \State $\mathbf{n=k}$:~Given $\mathrm{PAM}_{k-1}(m, \delta)$, the $k$-th vertex is added with $m$ edges $\{e_{k,\ell}\}_{\ell\in[m]}$. The edges connect sequentially to vertices of $\mathrm{PAM}_{k-1}(m, \delta)$ with probabilities proportional to an affine function of the current degree, which is updated after each new connection; in other words, for $\ell\in[m-1]$ and $v\in[k-1]$,
 \[
 \pr(e_{k,\ell+1} \text{ connects to } v)=
 \frac{\deg^{(\ell)}_{k-1}(v)+\delta}{2m(k-2) + \ell + (k - 1)\delta},
 \]
 where $\deg_{k-1}^{(\ell)}(v)$ denotes the degree of vertex $v\in \mathrm{PAM}_{k-1}(m, \delta)$ after $\ell$ edges from the $k$-th vertex have been connected to the graph.
\end{algorithmic}
The advantage of this variant is that it admits an exchangeable representation via a hidden environment given by a family of independent Beta-distributed random variables, see Section~\ref{sec:model}.

For PAM$_n(m,\delta)$, the expected number of triangles is well understood~\cite{BollobasRiordan, EggemannNoble, SubgraphsPAM} and is summarized below in Theorem~\ref{thm:ET}. The original results are stronger than what we state, as they also provide the constants in terms of $m$ and $\delta$. We observe that some of the authors technically considered different variations of the model, but analogous results can be easily recovered for the model above; in fact, we will present a short proof of this theorem in Section~\ref{sec:graphical_comp}.

\begin{thm}[{\cite{BollobasRiordan,EggemannNoble,SubgraphsPAM}}]\label{thm:ET}
 Let $m\ge2$ be a fixed integer and $\delta> -m$ be a fixed real number. Let $T_n$ be the number of triangles in $\pam(m,\delta)$. Then
\begin{align*}
 \E T_n\cong \begin{cases}
 \log n & \text{ if } \delta>0\\
 (\log n)^3 & \text{ if } \delta=0\\
 n^{|\delta|/(2m+\delta)}\log n & \text{ if } \delta\in(-m,0).
 \end{cases}
    \end{align*}
\end{thm}

Here and throughout the paper, $\cong$ denotes equality up to multiplying by some fixed constant that may depend on parameters of the model. In general, we use the big-$O$ and small-$o$ notation as in the usual definition, often replacing big-$O$ with $\lesssim$ for simplicity.

\subsection{Main results}
Our first main result explains the nature of the phase transition in the mean behavior from Theorem~\ref{thm:ET}, by interpolating between different parameter regimes. To be precise, we consider a more general model constructed with the same steps as $\pam$, but where the parameter $\delta$ also depends on the current size of the system. Formal definition is given in Section~\ref{sec:phase_transition}.

\begin{thm}\label{thm:phase_transition}
    Let $m\ge2$, $c\in\bR$, $\alpha>0$ be fixed and $\delta_k:=2mc\,(\log k)^{-\alpha}$ for $k\ge 2
    .$
    Let $T_n$ be the number of triangles in $\pam(m,(\delta_k)_{k\in[n]})$. Then
    \begin{align*}
\E T_{n} \cong\begin{cases}
(\log n)^{1+2\alpha} &\text{ when }  \alpha\in(0,1), c>0,\\
(\log n)^3 &\text{ when }\alpha>1 \text{ or }\alpha=1, c>-1,\\
(\log n)^{3}\log\log n &\text{ when }\alpha=1, c=-1,\\
 (\log n)^{2-c} &\text{ when }\alpha=1, c<-1\\
 (\log n)^{1+\alpha} e^{-\frac{c}{1-\alpha} (\log n)^{1-\alpha}} &\text{ when }\alpha\in (0,1), c<0.
\end{cases}
\end{align*}
In particular, fixing $\alpha=1$, the phase transition for $\E T_{n}$ occurs at $c=-1$.
\end{thm}

We believe that with techniques similar to ours, one can also show fluctuation phase transitions in the model. For example, in terms of fluctuations, fixing $\alpha=1$, the transition should occur at $c=1/2$, see open question~\ref{ques:phase_trans}.

Our second main result establishes distributional limits for the number of triangles in PAM$(m,\delta)$ and gives explicit bounds on the rates of convergence. This addresses a question raised in~\cite[page 10]{SubgraphsPAM}.
Similarly to the behavior of the mean given by Theorem~\ref{thm:ET}, there are three different regimes depending on the sign of $\delta$. In the case when $\delta>0$, the limiting distribution is normal, and the analysis required for the proof is completely different from that in the remaining cases.

To measure the rate of convergence, we will use the $L^{1}$-Wasserstein distance $\dwas$. For two random variables $X$ and $Y$, it is defined by
\[
\dwas(X,Y):=\sup_{h:\,\mathrm{Lip}(h)\le 1} \left|\E h(X)-\E h(Y)\right|.
\]

 Let $\psi_1\equiv 1$ and $\{\psi_i\}_{i\ge 2}$ be independent
 $\mathrm{Beta}(m+\delta,(2m+\delta)(i-1)-m)$ random variables. Define the nonnegative sequence of real numbers $(b_{i})_{i\ge 1}$ by
 \begin{align}\label{def:bi}
 b_{i}:=\begin{cases}
        1 &i=1\\
     \norm{\psi_i}_2^2\cdot \prod_{t=2}^{i} \norm{1-\psi_t}_{2}^{-2} &i\ge 2.
 \end{cases}
 \end{align}
 Observe that $(\psi_{i},b_{i})_{i\ge 2}$ depend on the parameters $m,\delta$ and define 
 \begin{align}
 s_n=
 \begin{cases} 
 \sqrt{\log n}, & \text{if } \delta>0,\\[1ex]
 \frac14(\log n)^2, & \text{if } \delta=0,\\[1ex]
 \left({m}/{|\delta|}-1\right)n^{|\delta|/(2m+\delta)}\log n, & \text{if } \delta\in(-m,0).
 \end{cases}
 \end{align}
 For $\delta \le 0$, we interpret $s_n$ as the expected number of triangles that contain the starting vertex~$1$. 

\begin{thm}\label{thm:main}
 Let $m\ge2$ and $\delta>-m$ be fixed. Let $T_n$ be the number of triangles in $\pam(m,\delta)$. We have $$\var(T_{n})\cong s_{n}^{2}.$$ Moreover, there exists a constant $\gamma =\gamma (m,\delta)\in (0,\infty)$ such that the following holds.
 \begin{enumerate}[\upshape (i)]
 \item\label{thm:main_delta_pos} When $\delta>0$, we have 
 \begin{align*}
 \dwas\left(\frac{T_n-\E T_n}{\gamma {s_n}}, Z\right)\lesssim \frac{1}{\sqrt{\log n}},
 \end{align*}
 where $Z\sim \N(0,1)$.
 \item\label{thm:main_delta_zero} When $\delta=0$, we have 
 \begin{align*}
 \dwas\left(\frac{T_n-\E T_n}{\gamma {s_n}}, 
\sum_{i=1}^\infty b_{i}\left(\frac{\psi_i^2}{\norm{\psi_i}_2^2}\prod_{j=i+1}^\infty \frac{(1-\psi_j)^2}{\norm{1-\psi_j}_2^2}-1\right)\right)
\lesssim \frac{1}{\sqrt{\log n}}.
 \end{align*}
 \item\label{thm:main_delta_neg} When $\delta\in(-m,0)$, we have 
\begin{align*}
 \dwas\left(\frac{T_n}{\gamma  s_n}, 
\sum_{i=1}^\infty b_{i}\cdot \frac{\psi_i^2}{\norm{\psi_i}_2^2}\prod_{j=i+1}^\infty \frac{(1-\psi_j)^2}{\norm{1-\psi_j}_2^2}\right)
\lesssim \frac{1}{\sqrt{\log n}}.
\end{align*}
 \end{enumerate}
\end{thm}
We present our interpretation of the limiting distribution in parts~\ref{thm:main_delta_zero} and~\ref{thm:main_delta_neg} alongside simulations in Subsection~\ref{sec:inerpr_lim}. 

\begin{rem}
    The value of the constant $\gamma$ is given by 
 \begin{equation}\label{eq:gamma}
     \gamma=\begin{cases}
         \frac{m}{\delta}\sqrt{\frac{(m-1)(m+\delta)(m+\delta+1)}{2m+\delta}}&\text{ if }\delta> 0,\\
         \qquad m^2(m-1) &\text{ if }\delta\le 0. 
     \end{cases}
 \end{equation}
\end{rem}
\begin{rem}[Optimality]
 Since for $\delta>0$ the rate of convergence in Theorem~\ref{thm:main} is of order $1/\sqrt{\var(T_n)}$, it is reasonable to expect it to be optimal, although it still requires establishing an anti-concentration result (or a local limit theorem). For $\delta\le 0$, our rate of convergence is far from the reciprocal of the standard deviation, and so we do not yet have a strong enough understanding to conjecture the optimal rate. Establishing optimality in the former case and improving the convergence rate in the latter case would be an interesting research direction. 
\end{rem}

\subsection{Our techniques}\label{sec:technique}
Establishing limiting distribution results for evolving graphs is often more challenging than for their static counterparts. There are two primary reasons for that. 

First, many common approaches for showing limiting distributions in discrete models seem to fail for preferential attachment models due to the global temporal dependence inherent in their evolution and the lack of exchangeability between edges. Thus, besides a few notable exceptions, the common techniques in this setting reduce to martingale convergence theorems and renewal-type arguments, both of which require careful control of relevant graph statistics. This is precisely the second primary reason for the difficulty of deriving distributional limits. The combinatorics get involved and require novel ideas even to derive the limiting behavior of the mean, \eg~see~\cite{BollobasRiordan,EggemannNoble,SubgraphsPAM}. Naturally, tackling higher moments multiplies the complexity of the task.

It is of particular interest to expand applications of other techniques, such as Stein's method, to non-static models. To prove distributional convergence via Stein's method, one usually needs to either establish control on the change in the statistic of interest under a perturbation in the system via some sort of coupling or derive a quantified bound on the dependency between the random variables, see~\cite{Ross11, ChenGoldstein11}.
However, in models such as $\pam(m,\delta)$, any small perturbation such as resampling a single edge may cause a significant change in the later evolution of the graph, and there is no clear way to bound the ``range" of dependence.
A recent successful example of such an application of this method in a non-static model is~\cite{BHJT25}, where the authors showed the central limit theorem for subgraph counts in the uniform attachment models by using the Stein--Chen method for Poisson approximation as an intermediate step. 
In their case, the model shares many properties with $\pam(m,\delta)$, except for the absence of bias in the attachment scheme. 
It is plausible that their approach could be adapted to $\pam(m,\delta)$ when $\delta>0$; however, the computations became extremely involved, and we concluded that the size-biased coupling is not the most natural approach.

With these challenges in mind, our approach has several key ideas that we would like to highlight, 
\begin{enumeratei}
    \item exchangeable description of the model, 
    \item centered-edge decomposition,
    \item graphical computation,
    \item applications of Stein's method and $L_2$-bounds.
\end{enumeratei}
First, as we already mentioned, the sequential attachment variant of $\pam(m,\delta)$ admits an exchangeable description in terms of hidden weights introduced in~\cite{urngraph}, see Subsection~\ref{sec:model}. In particular, almost all of the analysis is carried out conditioned on the values of these weights.
\begin{itemize}[leftmargin=*, itemsep=8pt]
    \item \textit{Centered-edge decomposition.} Next, we work with the centered-edge decomposition of the subgraph counts. The idea of working with centered subgraph counts has been particularly useful for simplifying covariance computations in various settings; see, \eg~\cite {KaurRollin21,DT_CCLT,TadasCLTclique,fang2025conditionalcentrallimittheorems}. Focusing on the number of triangles $T_n$ in some random graph on $[n]$, we can rewrite it as follows
\begin{align}
    T_n-\E T_n=\wt{T}_{n}^{(0)}+\wt{T}_{n}^{(1)}+\wt{T}_{n}^{(2)} + \wt{T}_{n}^{(3)},\label{eq:center_edge_intro}
\end{align}
where, denoting by $\ind_{i\leftrightarrow j}$ the indicator random variable that vertex $i$ is connected to $j$ and by $\overline{\ind_{i\leftrightarrow j}}$ its centered version, we have
\begin{align*}
&\wt{T}_{n}^{(0)}:=\sum_{i<j<k}\pr(i\leftrightarrow j)\pr(i\leftrightarrow k)\pr(j\leftrightarrow k), 
&&\wt{T}_{n}^{(1)}:=\sum_{i<j}\left(\sum_{k\neq i,j}\pr(i\leftrightarrow j)\pr(j\leftrightarrow k)\right)\overline{\ind_{i\leftrightarrow j}},\\
&\wt{T}_{n}^{(2)}:=\sum_{i<j}\pr(i\leftrightarrow j)\sum_{k\neq i,j} \overline{\ind_{i\leftrightarrow k}}\cdot\overline{\ind_{j\leftrightarrow k}}, &&\wt{T}_{n}^{(3)}=\sum_{i<j<k}\overline{\ind_{i\leftrightarrow j}}\cdot\overline{\ind_{i\leftrightarrow k}}\cdot\overline{\ind_{j\leftrightarrow k}}.
\end{align*}
Working with centered-edge counts has several advantages. Firstly, any two subgraphs that do not share any edges are automatically uncorrelated, which crucially simplifies the computation of the variances. Secondly, it highlights the competition between the terms $\wt{T}_{n}^{(0)}, \wt{T}_{n}^{(1)}, \wt{T}_{n}^{(2)},$ and $\wt{T}_{n}^{(3)}$ that otherwise might be overlooked. As we will see throughout this paper, depending on the parameter regime, different terms dominate, simplifying the analysis and providing intuition for the model's behavior and possible phase transitions. For example, in the case~\ref{thm:main_delta_pos} of Theorem~\ref{thm:main}, the dominating term is $\wt{T}_{n}^{(3)}$ while in the remaining cases it is $\wt{T}_{n}^{(0)}$. We interpret it as a competition between the mean-field interaction in the triangle formation and the effect of the hidden random environment. In particular, understanding the dominant term is key to establishing Theorem~\ref{thm:phase_transition}. 

In Appendix~\ref{sec:appendix_RGIV}, we illustrate the connection between the centered-edge decomposition and the classical Hoeffding decomposition in a different model of independent interest. Namely, we consider the random graph with immigrating vertices (RGIV), introduced by Aldous and Pittel in~\cite{RGIV}. The analysis of this model is significantly less involved than that of PAM, while the model still exhibits nontrivial competition among terms in the centered-edge decomposition of triangle counts, or even edge counts. Hence, using similar ideas, we establish limiting distributions for these counts in various regimes.

\item \textit{Graphical computation.} To address the second challenge of unraveling combinatorics, we developed an intuitive book-keeping mechanism, which we call the graphical computation lemma, see Section~\ref{sec:graphical_comp}. While it is possible to do all of the computations by hand, in our opinion, it would take already involved computations to another level, making them extremely difficult to follow. On the other hand, many of the sums one has to compute share a similar structure and geometric interpretation. Therefore, the main idea of this tool is to convert the problem into a diagram reading, which will essentially evaluate the sum. In particular, with this lemma, one can quickly derive Theorem~\ref{thm:ET} in all three regimes of parameters, again, ignoring the constants that the original results carefully computed. It can also be used to derive analogous results for other subgraphs such as cycles and cliques, see Lemma~\ref{lem:cycle}. Furthermore, by simplifying the computations in this way, we gain insight into the structure of the dominating terms. This led to much of the progress, even in results that technically do not explicitly rely on the graphical computation lemma, \eg~Theorem~\ref{thm:phase_transition}. Finally, it also enabled us to give detailed predictions for the limiting distributions of cycles and cliques given in Open Questions~\ref{ques:cycle} and~\ref{ques:cliques}.

\item \textit{Stein's method and $L_2$-techniques.} Finally, we approach case~\ref{thm:main_delta_pos} of Theorem~\ref{thm:main} in a different way than the other cases. When $\delta >0$, the mean-field interaction dominates the behavior, and we can apply Stein's method of exchangeable pairs. In the remaining cases~\ref{thm:main_delta_zero} and~\ref{thm:main_delta_neg}, we reduce the problem to a sum of martingale differences and then carefully analyze the error terms in the $L_2$-norm.
\end{itemize}

\subsection{Organization}
This paper is organized as follows. In Section~\ref{sec:support}, we present an exchangeable description of $\pam(m,\delta)$, the centered-edge decomposition for the number of triangles, and several supporting lemmas. Section~\ref{sec:graphical_comp} is dedicated to the graphical computation lemma. Section~\ref{sec:fluct_order}, establishing the dominating term from the centered-edge decomposition by computing the respective orders of fluctuation. Section~\ref{sec:clt} is dedicated to the proof of part~\ref{thm:main_delta_pos} of Theorem~\ref{thm:main}, while Section~\ref{sec:delta_neg} treats the remaining two parts of it.
Section~\ref{sec:phase_transition} covers the phase transition in the mean with respect to $\delta$ and the proof of Theorem~\ref{thm:phase_transition}.
Finally, we finish with closing remarks in Section~\ref{sec:closing}, where we discuss other variants of the preferential attachment models, other subgraph counts, and present simulations to interpret the limiting distribution from parts~\ref{thm:main_delta_zero} and~\ref{thm:main_delta_neg} of Theorem~\ref{thm:main}.

\section{Centered-edge decomposition}\label{sec:support}

\subsection{Urn graph and centered-edge decomposition}\label{sec:model}
We briefly review the construction of the P\'olya urn graph introduced in~\cite{urngraph}. For a more detailed explanation, we refer to~\cite[Chapter 5]{RemcoBook2}. Let $\psi_1\equiv 1$ and, for all $i\ge 2$, let $\psi_i$ be independent $\mathrm{Beta}(m+\delta,(2m+\delta)(i-1)-m)$ random variables and define
\begin{align*}
\varphi_{j,n}:=\psi_j\prod_{t=j+1}^n(1-\psi_t),\quad S_{(k,n]}:=\sum_{j=1}^k\varphi_{j,n}=\prod_{t=k+1}^n(1-\psi_t), \quad I_k^{(n)}:=\left(S_{(k-1,n]},S_{(k,n]}\right].
\end{align*}
Conditional on $\mvpsi:=(\psi_i)_{i\ge 1}$, let 
\begin{align*}
\{ U_{k,\ell}\}_{k\in[n],\ell\in[m]} \text{ be independent and uniformly distributed on } [0,S_{(k-1,n]}].
\end{align*}
For any $1\le j<k\le n$, place an edge from $k$ to $j$ for every $U_{k,\ell}$ contained in $I_j^{(n)}$ for some $\ell\in[m]$. By~\cite[Theorem 5.10]{RemcoBook2}, the random graph constructed in such a manner has the same distribution as the sequential linear preferential attachment model $\pam(m,\delta)$.

Let $\go_{ij}^{[\ell]}=\ind\{j\xrightarrow{(\ell)} i\}$ be the indicator function that vertex $j$ connected to vertex $i$ with the $\ell$--th edge. Define the number of triangles in $\pam(m, \delta)$ by
\begin{align*}
T_n=\sum_{1\le i<j<k\le n}\sum_{\ell_1,\ell_2\neq \ell_3 \in [m]}\go_{ij}^{[\ell_1]}\go_{ik}^{[\ell_2]}\go_{jk}^{[\ell_3]}.
\end{align*}
Notice that the conditional mean of every $\go_{ij}^{[\ell]}$ is the same for every $\ell\in[m]$,
\begin{align}\label{eq:theta}
\theta_{ij}:=\E\left(\go_{ij}^{[\ell]}\,\big|\,\mvpsi\right)=\psi_i\prod_{t=i+1}^{j-1}(1-\psi_t).
\end{align}
Then, converting to the centered-edge count, we can rewrite 
\begin{align}
T_n:=m^{2}(m-1)\cdot \Delta_{n}+ m(m-1)\cdot T_{n}^{(1)}+m\cdot T_{n}^{(2)} + T_{n}^{(3)},\label{eq:centered_decomp}
\end{align}
where 
\begin{align}
\Delta_{n}&=\Delta_{n}(\mvpsi) :=\sum_{i<j<k}\theta_{ij}\theta_{ik}\theta_{jk}\label{term:central0},\\
T_{n}^{(1)}&:=\sum_{i<j}\left(\sum_{k\neq i,j}\theta_{ik}\theta_{jk}\right)\sum_{\ell=1}^m\overline{\go}_{ij}^{[\ell]},\label{term:central1}\\
T_{n}^{(2)}&:=\sum_{i<j}\theta_{ij}\sum_{k\neq i,j}\sum_{\ell_1,\ell_2=1}^m \overline{\go}_{ik}^{[\ell_1]}\overline{\go}_{jk}^{[\ell_2]},\label{term:central2}\\
\text{and } T_{n}^{(3)}&=\sum_{i<j<k}\sum_{\ell_1,\ell_2,\ell_3=1\atop \ell_2\neq\ell_3}^m\overline{\go}_{ij}^{[\ell_1]}\overline{\go}_{ik}^{[\ell_2]}\overline{\go}_{jk}^{[\ell_3]}\label{term:central3}.
\end{align}
We can interpret these quantities as a conditional mean and as weighted center-–edge counts of edges, wedges, and triangles, respectively. We also note that this decomposition is slightly different from~\eqref{eq:center_edge_intro} because we pulled out the constant factors, but this change is largely inconsequential.
Next, we identify the order of fluctuations for the four quantities. 
\begin{thm}\label{thm:variance_order}
 The dominating term in the centered-edge decomposition of $T_n$, as in~\eqref{term:central0}--\eqref{term:central3}, is of fluctuation order 
 \begin{enumeratei}
 \item $\sqrt{\log n}$ and is given by $T_{n}^{(3)}$ when $\delta>0$,
 \item $(\log n)^2$ and is given by $\Delta_{n}$ when $\delta=0$,
 \item $n^{-\delta/(2m+\delta)}\log n$ and is also given by $\Delta_{n}$ when $\delta\in (-m,0)$.
 \end{enumeratei}
\end{thm}
The proof is given in the following four lemmas, which treat each term from~\eqref{term:central0}--\eqref{term:central3} separately in the reverse order.

\begin{lem}\label{lem:Tn3}
 For $m\ge 2$ and $\delta>-m$. Then
\begin{align*} 
\var(T_{n}^{(3)})\cong\begin{cases}
 \log n & \text{ if } \delta>0\\
 (\log n)^3 & \text{ if } \delta=0\\
 n^{-\delta/(2m+\delta)}\log n & \text{ if } \delta\in (-m,0)
 \end{cases}.
\end{align*}
\end{lem}

\begin{lem}\label{lem:quad_term}
 For $m\ge 2$ and $\delta>-m$. Then
 \begin{align*}
 \var(T_{n}^{(2)})\cong \begin{cases}
 1 & \text{ if } \delta>0\\
 \log n & \text{ if } \delta=0\\
 n^{-\delta/(2m+\delta)} & \text{ if } \delta\in (-m,0)
 \end{cases}.
 \end{align*}
\end{lem}

\begin{lem}\label{lem: lin term}
 For $m\ge 2$ and $\delta>-m$. Then
 \begin{align*}
 \var(T_{n}^{(1)})\cong \begin{cases}
 1 & \text{ if } \delta>0\\
 \log n & \text{ if } \delta=0\\
 n^{-2\delta/(2m+\delta)} & \text{ if } \delta\in (-m,0)
 \end{cases}.
 \end{align*}
\end{lem}

\begin{lem}\label{lem: centering term}
 For $m\ge 2$ and $\delta>-m$. Then
 \begin{align*}
 \var\left(\Delta_{n}\right)\cong \begin{cases}
 1 & \text{ if } \delta>0\\
 (\log n)^4 & \text{ if } \delta=0\\
 n^{-2\delta/(2m+\delta)}(\log n)^2 & \text{ if } \delta\in (-m,0)
 \end{cases}.
 \end{align*}
\end{lem}

\subsection{Auxiliary quantities}
Since our analysis becomes quite technical, we introduce several variables to simplify expressions and isolate statements that define their properties. Some of the computations below are similar to those from~\cite{SubgraphsPAM} and some of the references therein. We do not skip them for uniformity in presentation for the convenience of the reader.

First, it turns out to be more convenient to work with triangles indexed by their oldest vertex, hence we introduce the following quantities and write the sum from the conditional mean in~\eqref{term:central0} as
$
\Delta_n:=\sum_{i=1}^{n-2} \Delta_{n,i}
$ 
where 
\begin{align}\label{eq:Triangle_i}
\Delta_{n,i}:=\sum_{i<j<k\le n} \theta_{ij}\theta_{ik}\theta_{jk}
= \psi_{i}^{2}\cdot \sum_{i<j<k\le n} \prod_{t=i+1, t\neq j}^{k-1}(1-\psi_{t})^{2}\cdot \psi_{j} (1-\psi_{j}),\quad i\ge 1. 
\end{align}
Observe that $m^2(m-1)\cdot \Delta_{n, i}$ is the conditional expectation given $\mvpsi$ of the number of triangles with $i$ as the lowest labeled vertex. 
For $i,j,k\in[n]$ such that $i<j<k\le n$ define 
 \begin{align}\label{eq:mu_ijk}
 \mu_{i,j,k}&:=\E\theta_{ij}\theta_{ik}\theta_{jk}.
 \end{align}
We have
\begin{align}\label{eq:mu_idotdot}
\mu_{i,\cdot,\cdot}^{(n)}:=\E \Delta_{n,i}
&= \sum_{i<j<k\le n} \E\psi_{i}^{2}\cdot \prod_{t=i+1}^{k-1} \E(1-\psi_{t})^2 \cdot \frac{\E(\psi_j(1-\psi_j))}{\E(1-\psi_j)^2}\notag\\
&= \sum_{i<j<k\le n} \frac{\E\psi_{i}^2}{\prod_{t=2}^{i} \E(1-\psi_{t})^2}\cdot \prod_{t=2}^{k-1} \E(1-\psi_{t})^2 \cdot \frac{\E(\psi_j(1-\psi_j))}{\E(1-\psi_j)^2}.
\end{align}
Next, we identify the exact behavior of the $\mu_{i,\cdot,\cdot}^{(n)}$ terms upto a $(1+o(1))$-factor.  Since $\psi_t\sim\mathrm{Beta}(m+\delta,(2m+\delta)(t-1)-m)$, standard beta distribution arguments imply that
$
 (2m+\delta)\cdot t \psi_t 
$
converges in distribution to 
$\mathrm{Gamma}(m+\delta,1)$ \text{ as } $t\to\infty$.
 In particular, $t \psi_t$ is of constant order in all moments. The following lemma presents exact moment computations for $\psi_t,1-\psi_t$, which we will need later.

\begin{lem}\label{lem:psimom}
For $m\ge 2$ and $\delta>-m$, let $\beta:=(m+\delta)/(2m+\delta)$. Then for all $t\ge 1, k\ge 1$ we have
 \begin{align*}
 \E\psi_t^k &=\frac{\gC(m+\delta+k)}{\gC(m+\delta)}\cdot \frac{\gC((2m+\delta)t-2m)}{\gC((2m+\delta)t-2m+k)},\\
 \E(1-\psi_t)^k &= \frac{\gC((2m+\delta)(t-1)-m +k)}{\gC((2m+\delta)(t-1)-m)}\cdot \frac{\gC((2m+\delta)t-2m)}{\gC((2m+\delta)t-2m+k)}.
\end{align*}
In particular, we have $\norm{\theta_{ij}}_2 \cong \norm{\theta_{ij}}_1 \cong i^{\beta-1} j^{-\beta}.$
\end{lem}
Before we present the proof, let us first discuss some of the consequences. By Lemma~\ref{lem:psimom}, we have for all $k\ge 1, t\ge 2$,
\begin{align*}
    \prod_{i=2}^{t-1}\E(1-\psi_i)^k
    &=\prod_{s=0}^{k-1}\frac{\gC(1+(s+\delta)/(2m+\delta))}{ \gC(1+(s-m)/(2m+\delta))}\cdot \frac{ \gC(t-1+(s-m)/(2m+\delta))}{ \gC(t-1+(s+\delta)/(2m+\delta))}.
\end{align*}

We observe that similar computations appeared in~\cite{SubgraphsPAM} for the case $k=2$.
Next, we will explicitly identify the sequence $b_{i}$ from~\eqref{def:bi}.
Define $\beta:=(m+\delta)/(2m+\delta)$, so that $1/(2m+\delta)=(1-\beta)/m$. Define the sequence
\begin{align}\label{def:fn}
f_n:=f_n(\beta):=1+
\sum_{t=3}^{n} \prod_{i=2}^{t-1} \E(1-\psi_i)^2 
\end{align}
for $n\ge 3$ and $f_1=0,f_2=1$. In particular, we have 
\begin{align*}
 \Delta f_t := f_t - f_{t-1}
 &=\prod_{i=2}^{t-1} \E(1-\psi_i)^2 \\
 &= \frac{\gC(2\beta)\gC(2\beta+(1-\beta)/m)}{\gC(\beta)\gC(\beta+(1-\beta)/m)}\cdot \frac{\gC(t-2+\beta)\gC(t-2+\beta+(1-\beta)/m)}{\gC(t-2+2\beta)\gC(t-2+2\beta+(1-\beta)/m)}\cong t^{-2\beta}.
\end{align*}
It is easy to check that  
\begin{align*}
\mu_{i,j,k}&=
\frac{\E \psi_i^2}{\Delta f_{i+1}}\cdot \Delta f_k\cdot \frac{\E \psi_j(1-\psi_j)}{\E(1-\psi_j)^2}
\text{ and } 
b_{i}=\frac{\E \psi_i^2}{\Delta f_{i+1}} \cong i^{2\beta-2}
\end{align*}
for $1\le i<j<k\le n$. Moreover, we have 
\begin{align*}
 \frac{\Delta f_{i+1}}{\E \psi_i^2}\cdot\mu_{i,\cdot,\cdot}^{(n)}
 &= \sum_{j=i+1}^{n-1} \frac{\E \psi_j(1-\psi_j)}{\E(1-\psi_j)^2} (f_n-f_j).
\end{align*}
Define
\begin{align}
\tilde{s}_n&:= f_n\sum_{j=2}^{n-1} \frac{\E \psi_j(1-\psi_j)}{\E(1-\psi_j)^2}
- \sum_{j=2}^{n-1} \frac{\E \psi_j(1-\psi_j)}{\E(1-\psi_j)^2}\, f_j \notag\\
&= \bigl(1 +O(1/\log n)\bigr)\cdot
\begin{cases}
 \dfrac{\beta}{1-2\beta}n^{1-2\beta}\log n, & \text{if } 0<\beta<1/2,\\[2mm]
 \dfrac14(\log n)^2, & \text{if } \beta=1/2.
\end{cases}\label{eq:s_n}
\end{align}
Thus, we have $\abs{\tilde{s}_n/s_{n}-1}\lesssim 1/\log n$. Next, we present the proof of Lemma~\ref{lem:psimom}.

\begin{proof}[Proof of Lemma~\ref{lem:psimom}]
 Since for all $i>2$ the random variable $\psi_i$ has distribution $\mathrm{Beta}(m+\delta,m(2i-3)+\delta(i-1))$, the first two equalities follow from a direct computations as in~\cite[Eq. (5.7) and (5.9)]{SubgraphsPAM}.
 To compute $\norm{\theta_{ij}}_2$ first observe that
 \begin{align*}
 L_{ij}:=\log \frac{\Delta f_j}{\Delta f_{i+1}} &=\sum_{t=i+1}^{j-1}\log\left(\frac{\left(t-\frac{3m+\delta}{2m+\delta}\right)\left(t-\frac{3m+\delta-1}{2m+\delta}\right)}{\left(t-\frac{2m}{2m+\delta}\right)\left(t-\frac{2m-1}{2m+\delta}\right)}\right)\\
 &=-\sum_{t=i+1}^{j-1}\frac{2(m+\delta)}{(2m+\delta)t}+O\left(\sum_{t=i+1}^{j-1}t^{-2}\right)
 =-2\beta\sum_{t=i+1}^{j-1}t^{-1}+O\left(1/i\right).
 \end{align*}
 Therefore,
 \begin{align*}
 \norm{\theta_{ij}}^2_2=\E \psi_i^2\cdot\prod_{t=i+1}^{j-1}(1-\psi_t)^2
 =\E \psi_i^2\cdot \frac{\Delta f_j}{\Delta f_{i+1}} 
 &=\frac{(m+\delta)(m+\delta+1)}{(2m+\delta)^2}\cdot\frac{1}{i^2}\left(\frac{i}{j}\right)^{2\beta}(1+O(1/i)).
 \end{align*}
 This completes the proof.
 \end{proof}
The following result gives the asymptotic behavior of $\mu_{i,\cdot,\cdot}^{(n)}$ up to a $(1+o(1))$-factor.
\begin{lem}\label{lem:mu_ijk}
For $m\ge 2$ and $\delta\in(-m,0]$, let $\beta:=(m+\delta)/(2m+\delta)$. Then for each fixed $i\ge 1$,
\begin{align*}
\abs{\mu_{i,\cdot,\cdot}^{(n)}/\tilde{s}_n - b_i}\le C\cdot \frac{b_{i}\log i}{\log n},
\end{align*}
where $b_{i}$ is as in~\eqref{def:bi} and $C$ is an absolute constant. Moreover, we have
\[
\mu_{i,\cdot,\cdot}^{(n)} \lesssim \begin{cases}
    i^{-1} &\text{ if } \delta >0,\\
    \tilde{s}_n \cdot i^{2\beta-2} & \text{ if } -m<\delta \le 0.
\end{cases}
\]
\end{lem}
\begin{proof}[Proof of Lemma~\ref{lem:mu_ijk}]
Recall that for $1\le i<j<k\le n$,
\begin{align*}
\mu_{i,j,k}
=\E\psi_i^2\cdot \prod_{t=i+1}^{k-1} \E(1-\psi_t)^2\cdot \frac{\E(\psi_j(1-\psi_j))}{\E(1-\psi_j)^2}
= \frac{\E \psi_i^2}{\Delta f_{i+1}}\cdot \Delta f_k\cdot \frac{\E \psi_j(1-\psi_j)}{\E(1-\psi_j)^2}.
\end{align*}
Hence, 
\begin{align*}
\frac{\Delta f_{i+1}}{\E \psi_i^2}\,\mu_{i,\cdot,\cdot}^{(n)}
=\sum_{j=i+1}^{n-1} \frac{\E \psi_j(1-\psi_j)}{\E(1-\psi_j)^2}(f_n-f_j).
\end{align*}
Using
\begin{align*}
\frac{\E \psi_j(1-\psi_j)}{\E(1-\psi_j)^2}=\frac{\beta}{j}(1+o(1)),
\end{align*}
along with the asymptotics for $f_n$ and comparing with the definition of $\tilde{s}_n$, we obtain
\begin{align*}
\left|\frac{\Delta f_{i+1}}{\E \psi_i^2}\,\mu_{i,\cdot,\cdot}^{(n)}-\tilde{s}_n\right|
\le C\,f_n\log i.
\end{align*}
Dividing by $\tilde{s}_n$ and using $b_{i}=\E\psi_i^2/\Delta f_{i+1}$ completes the proof. 

For the upper bound on $\mu_{i,\cdot,\cdot}^{(n)}$ when $\delta \in (-m,0]$, so that $\beta\in (0,1/2]$, we use $f_n-f_j\le f_n$,
\begin{align*}
    C_1&:=\sup_{i\ge 2} i \E \psi_i/\E(1-\psi_i)^2 <\infty,\\
    C_2&:=\sup_{i\ge 2}i^2\E \psi_i^2<\infty,\\
    \text{and } C_3&:=\inf_{i\ge 2}i^{2\beta}\Delta f_i \in(0,\infty)
\end{align*}
to get
\[
\mu_{i,\cdot,\cdot}^{(n)} \le C_1 f_n\log n\cdot \frac{\E\psi_i^2}{\Delta f_{i+1}} \le \frac{C_1 C_2}{C_3} \cdot f_n\log n\cdot i^{2\beta-2}.
\]
When $\delta >0$ or $\beta\in (1/2,1)$, we use the fact that $f_n-f_j \le C_4 \cdot j^{1-2\beta}$ to get 
$
\mu_{i,\cdot,\cdot}^{(n)} \le C_1C_4\sum_{j>i}j^{-2\beta}\cdot \E\psi_i^2/\Delta f_{i+1} \le C_1 C_2C_4/C_3 \cdot i^{-1}.
$
\end{proof}

It will be helpful to consider normalized versions (mean one) of several random variables, which we introduce now 
\begin{align}\label{eq:XYU}
 X_t&:=\frac{\psi_t^2}{\E\psi_t^2},\quad 
 Y_t=\frac{(1-\psi_t)^2}{\E(1-\psi_t)^2},\quad 
 U_t=\frac{\psi_t(1-\psi_t)}{\E\psi_t(1-\psi_t)}, \quad \text{ and } \quad
 Y_t^{[j]}=\begin{cases}Y_t&\text{ if } j\neq t\\
 U_t &\text{ if } j= t\end{cases}.
\end{align}

Most computations will be carried out for various products and sums of the expectations of these quantities. We dedicate the rest of the subsection to supporting lemmas.

\begin{lem}\label{lem:bnds on Y}
 For the random variables given in~\eqref{eq:XYU}, some universal constant $C>0$, and all $1<t<t'$, we have 
 \begin{align*}
 \var(X_t)\le C,\quad\var(Y_t)\le \frac{C}{t^2},\quad\var(U_t)\le C, \quad \cov(Y_t,U_t)\le \frac{C}{t}, \quad \text{and}\quad  \prod_{i=t}^{t'}\E Y_i^2 \le C.
 \end{align*}
\end{lem}
\begin{proof}
The first four bounds follow from the fact that, as $t\to\infty$, $t\cdot \psi_t$ converges in distribution and in $L^p,p\ge 1$ norm to $\xi_{m+\delta }/\E\xi_{2m+\delta }$ where $\xi_m\sim\text{Gamma}(m,1)$. For the last bound, observe that $\E Y_t^2\ge (\E Y_t)^2=1$ for all $t>1$. Therefore, it follows from the fact that $\prod_{t=2}^\infty\E Y_t^2 
=\prod_{t=2}^\infty (1+\var(Y_t))\le \prod_{t=2}^\infty (1+C/t^2)\le \exp(C\sum_{t=2}^\infty t^{-2}) <\infty.$
\end{proof}

\begin{lem}\label{lem:powerk}
 For a random variable $\go$ supported on $(0,1)$ and for $k\ge 1$, we have
 \begin{align*}
 \abs{\log \E(1-\go)^k + k\E \go} \le \max\{k(k-1)/2, k/(1-\E \go)\}\cdot \E\go^2.
 \end{align*}
\end{lem}
\begin{proof}
 We have 
 \begin{align*}
 \E(1-\go)^k \le \E\left(1-k\go + \frac{1}{2}k(k-1)\go^2\right)\le \exp\left(-k\E\go + \frac{1}{2}k(k-1)\E\go^2\right)
 \end{align*} 
 and 
 \begin{align*}
 \E(1-\go)^k\ge (1-\E \go)^k \ge \exp\left(-\frac{k\E\go}{1-\E\go}\right) = \exp\left(-k\E\go - \frac{k(\E\go)^2}{1-\E\go}\right).
 \end{align*}
 This completes the proof.
\end{proof}

\section{Graphical computation}\label{sec:graphical_comp}

We would like to reiterate that random variables $\theta_{ij}$'s are not independent; therefore, when computing the expectation of products, one needs to rewrite these terms in terms of corresponding products of $\psi_{\cdot}$ or $(1- \psi_{\cdot})$ and redo the computation similar to the one above. In this particular example, using Lemma~\ref{lem:powerk} we get (assuming without loss of generality that $j<k$)
 \begin{align*}
 \E\theta_{ij}\theta_{ik}&=\E\psi_i^2\left(\prod_{t=i+1}^{j-1}\E (1-\psi_t)^2\right)\prod_{t=j}^{k-1}\E(1-\psi_t)
 \cong \frac{1}{i^2}\left(\frac{i}{j}\right)^{2\beta}\left(\frac{j}{k}\right)^{\beta}=i^{2\beta-2}j^{-\beta}k^{-\beta}.
 \end{align*}

To derive the expected values and the variances of various subgraph counts in PAM$(m,\delta)$, one needs to compute the sum of such expectations of products. In order to significantly simplify proofs, we introduce a ``graphical computation" lemma (see Lemma~\ref{lem:graphic}) as a convenient bookkeeping mechanism. To make the exact statement intuitive, we find it helpful to first consider the following computation, in which we derive the sum in a brute-force fashion.
\begin{align}
 \sum_{i<j<k}\E \theta_{ij}^2\theta_{ik}\theta_{jk}&=\E\psi_i^3\left(\prod_{t=i+1}^{j-1}(1-\psi_t)^3\right)(1-\psi_j)\psi_j\prod_{t=j}^{k-1}(1-\psi_t)^2\notag\\
 [\text{by Lemma}~\ref{lem:powerk}] \quad&\cong\sum_{i<j<k} \frac{1}{i^3}\left(\frac{i}{j}\right)^{3\beta}
 \frac{1}{j}\left(\frac{j}{k}\right)^{2\beta}\cong\sum_{i<j<k}\frac{1}{i^{3-3\beta}j^{1+\beta}k^{2\beta}}.\label{eq:example1}
\end{align}
We now would like sum first over $k$ (which ranges from $j$ till $n$), then over $j$ (which ranges from $i$ till $n$), and finally over $i$ (which ranges from $1$ till $n$). However it is clear that the answer will depend on the value of $\beta$ as in each step we compare the power of the variable with $1$. Hence, we consider the following three cases.

\textbf{Case of $\beta\in(1/2,1)$}: $\displaystyle
 \sum_{i<j<k}\frac{1}{i^{3-3\beta}j^{1+\beta}k^{2\beta}}\cong\sum_{i<j}\frac{1}{i^{3-3\beta}j^{3\beta}}\cong \sum_{i=1}^n \frac{1}{i^{2}}\cong 1.
$

\textbf{Case of $\beta= 1/2$}: $\displaystyle
 \sum_{i<j<k}\frac{1}{i^{3-3\beta}j^{1+\beta}k^{2\beta}}=\sum_{i<j<k}\frac{1}{i^{\frac{3}{2}}j^{\frac{3}{2}}k}\cong\log n\sum_{i<j}\frac{1}{i^{\frac{3}{2}}j^{\frac{3}{2}}}\cong \log n.
$

\textbf{Case of $\beta\in (0,1/2)$}: $\displaystyle
 \sum_{i<j<k}\frac{1}{i^{3-3\beta}j^{1+\beta}k^{2\beta}}\cong \sum_{i<j}\frac{n^{1-2\beta}}{i^{3-3\beta}j^{1+\beta}}\cong \sum_{i=1}^n\frac{n^{1-2\beta}}{i^{3-2\beta}}\cong n^{1-2\beta}.
$

Therefore
\begin{align}\label{eq:example1_1}
 \sum_{i<j<k}\E \theta_{ij}^2\theta_{ik}\theta_{jk}&=\begin{cases}
 1 &\text{ if } \beta>1/2\\
 \log n &\text{ if } \beta=1/2\\
 n^{1-2\beta}&\text{ if } \beta<1/2
 \end{cases}.
\end{align}

Indeed, each time when we sum over a variable we compare its power, denote it by $\gamma $, with $1$ (or with $-1$ if it is in the numerator) and if $\gamma >1$ it contributes a constant factor and $(\gamma -1)$ to the power of the next variable, if $\gamma =0$ it gives a $\log n$ factor, while if $\gamma <1$ it contributes $n^{\abs{\gamma -1}}$ factor. Therefore, we formalize this computation in the following general lemma.

\begin{lem}[Graphical computation]\label{lem:graphic}
 For $m\ge 2$ and $\delta>-m$ let $\beta:=(m+\delta)/(2m+\delta)$.
 Let $H=([k],E)$ be a finite connected (undirected) multigraph that does not contain self-loops but might contain multi-edges. Then the expectation of the $H$-subgraph count in $\pam(m,\delta)$ is given by
 \begin{align}
 &\E \sum_{1\le i_1< i_2< \cdots< i_k\le n} \prod_{(a,b)\in E\atop a>b}\theta_{i_b i_a}\cong 
 F_k\biggl(d_1-\beta \ell_1,d_2-\beta (\ell_2-\ell_1),\ldots, d_{k}-\beta(\ell_{k}-\ell_{k-1})\biggr),\label{eq:graphic_fla}
 \end{align}
 where
 \begin{itemize}
 \item $d_j$ is the in-degree of vertex $j$ (\ie~the number of edges $(j,k)\in E$ with $k>j$),
 \item $\ell_j$ is the number of edges $(u,v)\in E$ such that $v\le j< u$,
 \item $F_k(a_1,a_2,\ldots,a_k):=
 \sum_{1\le i_1< i_2< \cdots< i_k\le n} 
 \left(\prod_{j=1}^k i_j^{a_j}\right)^{-1}$.
 \end{itemize}

 Moreover, the function $F_k$ satisfies
 \begin{align*}
 F_k(&a_1,a_2,\ldots,a_k) \\
 &\cong F_{t-1}(a_1,a_2,\ldots,a_{t-1})\cdot
 \begin{cases}
 n^{k-t+1-(a_t+a_{t+1}+\cdots+a_k)} & \text{ if } a_t+a_{t+1}+\cdots+a_{k}< k-t+1\\
 \log n &\text{ if } a_t+a_{t+1}+\cdots+a_{k}= k-t+1
 \end{cases},
 \end{align*}
 where $t:=\max\{i\in [k]\mid a_i+a_{i+1}+\cdots+a_{k}\le k-i+1\}$ and 
 \begin{align*}
 F_k(a_1,a_2,\ldots,a_k) 
 \cong 1 \text{ if no such $t$ exists.}
 \end{align*}
\end{lem}

Note that the function $F_k$ also depends on $n$, which we omit from the notation as we do with many quantities throughout the paper. Before delving into the proof, we present a few applications.

\begin{lem}\label{cor:cycles_order}
 Suppose $m\ge2$ and $\delta>0$. Let $H=([k], E)$ be a finite connected multigraph such that all vertices have degrees at least two. Then the expectation of the $H$-subgraph count in $\pam(m,\delta)$ is
 \begin{enumerate}[\upshape (i)]
 \item\label{cor:cycles_order_part1} $\Theta(\log n)$ if $H$ is a cycle,
 \item\label{cor:cycles_order_part2} $\Theta(1)$ otherwise. 
 \end{enumerate}
\end{lem}
\begin{proof}
 Recalling the notations $\{(d_i,\ell_i)\}_{i\in [k]}$ from Lemma~\ref{lem:graphic} notice that
 \begin{align*}
 \ell_i-\ell_{i-1}=d_{\mathrm{in}}(i)-d_{\mathrm{out}}(i)=2d_i-\deg_H(i),
 \end{align*}
 where $d_{\mathrm{in}}(i)=d_i$ (resp.~$d_{\mathrm{out}}(i)$) denotes the in-degree (resp.~out-degree) in the orientation of edges of $H$ from the endpoint with the higher index to the one with the lower one, and $\deg_H(i)$ denotes the degree of $i$ in $H$. Consider any $t>1$, then summing over $i$ in $\{t,t+1,\ldots,k\}$ we get 
 \begin{align*}
 -\ell_{t-1}&=\sum_{j=t}^k2d_i-\deg_H(i)=2\left(\sum_{j=t}^kd_i-\frac{1}{2}\deg_H(i)\right).
 \end{align*}
 In the next expression, we use the facts that $\deg_H(i) \ge 2$ for all $i$, that $H$ is connected (implying $\ell_i > 0$ for every $i \in [k-1]$), and that $\delta > 0$, which in turn yields $\beta > 1/2$. Using these observations, we derive that for all $t>1$
 \begin{align*}
 \sum_{j=t}^{k} (d_j-\beta(\ell_j-\ell_{j-1}))=\sum_{j=t}^{k-1}d_j+\beta\ell_{t-1}&=(k-t+1)+\sum_{j=t}^{k-1}(d_j-1)+\beta\ell_{t-1}\\
 &\ge(k-t+1)+\sum_{j=t}^{k-1}\left(d_j-\frac{1}{2}\deg_H(i)\right)+\beta\ell_{t-1}\\
 &=(k-t+1)+\left(\beta-\frac{1}{2}\right)\ell_{t-1}>k-t+1.
 \end{align*}
 Finally, the case when $t=1$ is treated as follows
 \begin{align*}
 \sum_{j=1}^k (d_j-\beta(\ell_j-\ell_{j-1}))-k=\sum_{j=1}^k d_j-k=\abs{E(H)}-\abs{V(H)},
 \end{align*}
 which equals zero when $H$ is a cycle and is strictly positive when $H$ is any other connected multigraph with the lowest degree at least two.
 The conclusion now follows from Lemma~\ref{lem:graphic}.
\end{proof}

The following examples illustrate why we think of this lemma as a graphical computation and present a simple way to compute the parameters $\{(d_i,\ell_i)\}_{i\in [k]}$.

\begin{ex}\label{ex:example1}
 Returning to the example from~\eqref{eq:example1} we first represent $\theta_{ij}^2\theta_{ik}\theta_{jk}$ graphically, by placing variables on a segment respecting their order and adding a directed edge corresponding to each $\theta_{\cdot,\cdot}$ as shown in Figure~\ref{fig:Ex3.3}.
\begin{figure}[htb]
\begin{center}
\begin{tikzpicture}[thick, scale=0.8]
 \node[big dot,label=below:$i$] (i) at (-8,0) {};
 \node[big dot,label=below:$j$] (j) at (-5,0) {};
 \node[big dot,label=below:$k$] (k) at (-2,0) {};
 \draw [line width=0.5mm] (i) -- (k);

 \draw [<-,red] (i) to [bend left=50](j);
 \draw [<-,red] (i) to [bend left=30](j);
 \draw [<-,red] (j) to [bend left=40](k);
 \draw [<-,red] (i) to [bend left=50](k);

 \draw[dashed] (-8,-0.2) rectangle (-7.5,0.8);
 \draw (-8,1) node[above] {$\ell_i=3$};
 \draw (-8,-1) node[below] {$d_i=3$};
 
 \draw[dashed] (-5,-0.2) rectangle (-4.5,1.7);
 \draw (-5,2) node[above] {$\ell_j=2$};
 \draw (-5,-1) node[below] {$d_j=1$};
 
 \draw (-2,-1) node[below] {$d_k=0$};
 \draw (-2,1) node[above] {$\ell_k=0$};
 \end{tikzpicture}	
\end{center}
\caption{Graphical computation diagram for Example~\ref{ex:example1}.\label{fig:Ex3.3}}
\end{figure}
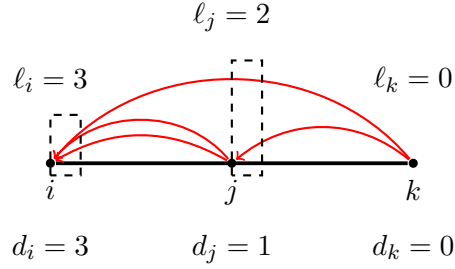 

Observe that $\ell_i$ is the number of edges that intersect the corresponding dashed rectangle.
Then Lemma~\ref{lem:graphic} yields that
\begin{align*}
 \sum_{i<j<k}\E \theta_{ij}^2\theta_{ik}\theta_{jk}
 \cong F_3(3-3\beta,1+\beta,2\beta)
 &=
 \begin{cases}
 1&\text{ if } \beta>1/2\\
 F_2(3/2,3/2) \cdot \log n &\text{ if } \beta=1/2\\
 F_2(3+3\beta,1+\beta) \cdot n^{1-2\beta}&\text{ if } \beta<1/2
 \end{cases}\\
 &=
 \begin{cases}
 1 &\text{ if } \delta>0\\
 \log n &\text{ if } \delta=0\\
 n^{-\delta/(2m+\delta)}&\text{ if } \delta<0
 \end{cases},
\end{align*}
recovering the answer from~\eqref{eq:example1_1}.
\end{ex}
 
\begin{ex}\label{ex:example2}
Consider a directed graph $H$ on six vertices as presented in Figure~\ref{fig:Ex3.4_1}. To compute $\E \sum_{1\le i_1< i_2< \cdots< i_k\le n} \prod_{(a\to b)\in E}\theta_{i_b i_a}$ we first depict $H$ as shown in Figure~\ref{fig:Ex3.4_2}, this makes it easier to compute the pairs ${(d_i,\ell_i)}_{i\in[6]}$. 
\begin{figure}[htb]
\begin{center}
\begin{tikzpicture}[thick, scale=0.7]
 \node[big dot,label=above:$1$] (j1) at (5,8) {};
 \node[big dot,label=above:$2$] (j2) at (7,7) {};
 \node[big dot,label=above:$3$] (j3) at (7,9) {};
 \node[big dot,label=above:$4$] (j4) at (3,7) {};
 \node[big dot,label=above:$5$] (j5) at (9,8) {};
 \node[big dot,label=above:$6$] (j6) at (11,7) {};
 \draw [-,red] (j1) to (j4);
 \draw [-,red] (j1) to (j3);
 \draw [-,red] (j1) to (j2);
 \draw [-,red] (j2) to (j4);
 \draw [-,red] (j2) to (j6);
 \draw [-,red] (j3) to (j5);
 \draw [-,red] (j5) to (j6);
 \end{tikzpicture}	
\end{center}
\caption{A graph from Example~\ref{ex:example2}.\label{fig:Ex3.4_1}}
\end{figure}
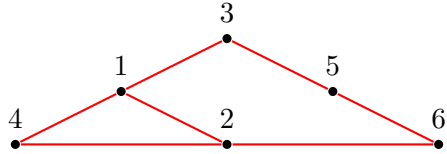 
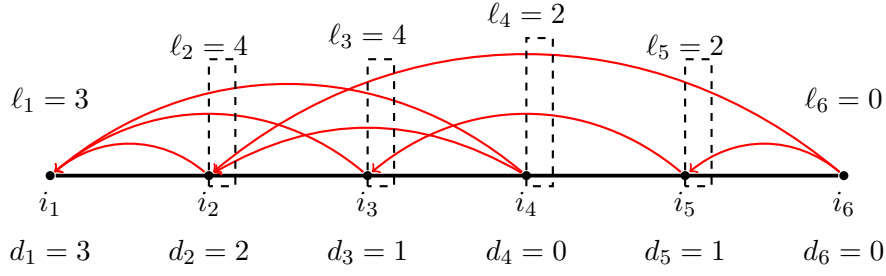
\begin{figure}[htb]
\begin{center}
\begin{tikzpicture}[thick, scale=0.7] 
 \node[big dot,label=below:$i_1$] (i1) at (0,0) {};
 \node[big dot,label=below:$i_2$] (i2) at (3,0) {};
 \node[big dot,label=below:$i_3$] (i3) at (6,0) {};
 \node[big dot,label=below:$i_4$] (i4) at (9,0) {};
 \node[big dot,label=below:$i_5$] (i5) at (12,0) {};
 \node[big dot,label=below:$i_6$] (i6) at (15,0) {};
 \draw [line width=0.5mm] (i1) -- (i6);

 \draw [<-,red] (i1) to [bend left=40](i4);
 \draw [<-,red] (i1) to [bend left=40](i3);
 \draw [<-,red] (i1) to [bend left=40](i2);
 \draw [<-,red] (i2) to [bend left=30](i4);
 \draw [<-,red] (i2) to [bend left=40](i6);
 \draw [<-,red] (i3) to [bend left=40](i5);
 \draw [<-,red] (i5) to [bend left=40](i6);

 \draw (0,1) node[above] {$\ell_1=3$};
 \draw (0,-1) node[below] {$d_1=3$};
 
 \draw[dashed] (3,-0.2) rectangle (3.5,2.2);
 \draw (3,2) node[above] {$\ell_2=4$};
 \draw (3,-1) node[below] {$d_2=2$};
 
 \draw[dashed] (6,-0.2) rectangle (6.5,2.2);
 \draw (6,-1) node[below] {$d_3=1$};
 
 \draw (6,2.2) node[above] {$\ell_3=4$};
 \draw[dashed] (9,-0.2) rectangle (9.5,2.6);
 \draw (9,2.6) node[above] {$\ell_4=2$};
 \draw (9,-1) node[below] {$d_4=0$};
 
 \draw[dashed] (12,-0.2) rectangle (12.5,2.2);
 \draw (12,2) node[above] {$\ell_5=2$};
 \draw (12,-1) node[below] {$d_5=1$};
 
 \draw (15,1) node[above] {$\ell_6=0$};
 \draw (15,-1) node[below] {$d_6=0$};
 \end{tikzpicture}	
\end{center}
\caption{Graphical computation diagram for Example~\ref{ex:example2}.\label{fig:Ex3.4_2}}
\end{figure} 

Then Lemma~\ref{lem:graphic} yields that
\begin{align*}
 \E \sum_{1\le i_1< i_2< \cdots< i_k\le n} \prod_{(a\to b)\in E}\theta_{i_b i_a}&\cong F_6\left(3-3\beta,2-\beta,1,2\beta,1,2\beta\right)\\
 &\cong \begin{cases}
 1 &\text{ if } \beta>1/2\\
 F_2(3/2,3/2)\cdot F_1(1)^4 &\text{ if } \beta=1/2\\
 F_2(3-3\beta,2-\beta)\cdot F_1(1)^2 \cdot F_1(2\beta)^2 &\text{ if } \beta<1/2
 \end{cases}.
\end{align*}
\end{ex}
Such computations are very common in the computation of expected values; in particular, Lemma~\ref{lem:graphic} allows us to recover Theorem~\ref{thm:ET}. Recall that the results in~\cite{BollobasRiordan,EggemannNoble,SubgraphsPAM} also give the constants in terms of $m$ and $\delta$, which we are ignoring. Careful tracking of the constants in our computations would indeed recover their values; however, this would entail repeating the computations of the original authors. 

\begin{proof}[Proof of Theorem~\ref{thm:ET}]
 From the centered-edge decomposition as in~\eqref{term:central0}--\eqref{term:central3} one can see that $ \E T_n\cong\E\sum_{i<j<k}\theta_{ij}\theta_{ik}\theta_{jk}.$
 In this case, the graphical computations are simple and given in Figure~\ref{fig:triangle_th1.1}.
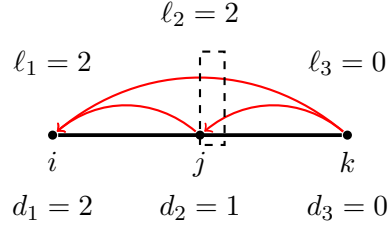
\begin{figure}[htb]
\begin{center}
\begin{tikzpicture}[thick, scale=0.65] 
 \node[big dot,label=below:$i$] (1) at (0,0) {};
 \node[big dot,label=below:$j$] (2) at (3,0) {};
 \node[big dot,label=below:$k$] (3) at (6,0) {};
 \draw [line width=0.5mm] (1) -- (3);

 \draw [<-,red] (1) to [bend left=40](2);
 \draw [<-,red] (1) to [bend left=40](3);
 \draw [<-,red] (2) to [bend left=40](3);

 \draw (0,1) node[above] {$\ell_1=2$};
 \draw (0,-1) node[below] {$d_1=2$};
 
 \draw[dashed] (3,-0.2) rectangle (3.5,1.7);

 \draw (3,2) node[above] {$\ell_2=2$};
 \draw (3,-1) node[below] {$d_2=1$};
 
 \draw (6,-1) node[below] {$d_3=0$};
 \draw (6,1) node[above] {$\ell_3=0$};
 \end{tikzpicture}	
\end{center}
\caption{Graphical computation diagram a triangle.\label{fig:triangle_th1.1}}
\end{figure}
By Lemma~\ref{lem:graphic}, recalling that $\beta:=(m+\delta)/(2m+\delta)$, we have
\begin{align*}
 \E T_n
 &\cong F_3(2-2\beta,1,2\beta)\cong\begin{cases}
 \log n &\text{ if } \beta >1/2\\
 (\log n)^3&\text{ if } \beta =1/2\\
 n^{1-2\beta}\log n &\text{ if } \beta <1/2 
 \end{cases}
 =\begin{cases}
 \log n & \text{ if } \delta>0\\
 (\log n)^3 & \text{ if } \delta=0\\
 n^{-\delta/(2m+\delta)}\log n & \text{ if } \delta\in(-m,0)
 \end{cases}.
\end{align*}
This completes the proof.
\end{proof}

The growth of the expected number of cycles and cliques can be derived from existing results in the literature (see~\cite{BollobasRiordan} for $\ell$-cycle when $\delta =0$ and the variational description in~\cite{SubgraphsPAM} for the general case). Here, we present a direct proof of these asymptotics using the graphical computation lemma. The proof also suggests a possible limit in the $\delta <0$ case.
\begin{lem}[Cycles and cliques]\label{lem:cycle}
Fix an integer $\ell \ge 4$. Let $C_\ell(n)$ and $K_\ell(n)$ denote, respectively, the number of $\ell$-cycles and $\ell$-cliques in the preferential attachment model $\mathrm{PAM}_{n}(m,\delta)$. Then the expectations $\E C_\ell(n)$ and $\E K_\ell(n)$ satisfy the following asymptotic regimes:

\begin{equation*}
    \E C_\ell(n)\cong \begin{cases}
 1 & \text{ if } \delta>0\\
 (\log n)^{\ell} & \text{ if } \delta=0\\
 n^{(1-2\beta)\lfloor \ell/2\rfloor} (\log n)^{\ell-2\lfloor \ell/2\rfloor} & \text{ if } \delta\in(-m,0)
 \end{cases},
\end{equation*}
and
\begin{equation*}
    \E K_\ell(n)\cong \begin{cases}
 1 & \text{ if } \delta>-\frac{m(\ell-3)}{\ell-2}\\[1em]
 \log n & \text{ if } \delta=-\frac{m(\ell-3)}{\ell-2}\\[1em]
 n^{1-\beta(\ell-1)} & \text{ if } \delta\in\left(-m,-\frac{m(\ell-3)}{\ell-2}\right)
 \end{cases},
\end{equation*}
where $\beta:=(m+\delta)/(2m+\delta)$.
\end{lem}

\begin{proof}[Proof of Lemma~\ref{lem:cycle}]
We focus on the $\delta \le 0$ or $\beta\le 1/2$ case, as $\delta >0$ follows from Lemma~\ref{cor:cycles_order}.
For $\ell$-clique we have $\E K_{\ell}(n)\cong F_{\ell}(\mvx)$ where $\mvx=(x_{1},x_{2},\ldots,x_\ell)$ where $x_i=(\ell-i) -\beta(\ell+1-2i)=(1-2\beta)(\ell-i) +\beta(\ell-1), {i\in[\ell]}$. Notice that  $x_{\ell}=\beta(\ell-1)$ and 
\begin{align*}
\sum_{j=\ell-i}^{\ell-1}x_{j} =\sum_{j=1}^{i}  ((1-2\beta)j +\beta(\ell-1))
& = (1-2\beta)i(i+1)/2 +\beta(\ell-1)i \\
&= i+i\cdot \left( \beta(\ell-1) -1+(1-2\beta)(i+1)/2\right)\\
&\ge  i+i\cdot \left( \beta(\ell-1) -1+1-2\beta\right) > i
\end{align*}
 for all $i\ge 1$.  Thus,  $F_{\ell}(\mvx)\cong F_{1}(x_\ell)\cong 1+ \log n \cdot \ind_{\beta=1/(\ell-1)} + n^{1-\beta(\ell-1)}\cdot \ind_{\beta<1/(\ell-1)}$.
 
 For the $\ell$-cycle, we have to separate into cases depending on the vertex embedding. Let us fix an embedding, that is, a subgraph on the vertex set $[\ell]$. Any vertex can have in-degree only from the set $\{0,1,2\}$. Let $a_{i}$ be the number of vertices with in-degree $i$ for $i=0,1,2$. We have $a_{0}+a_{1}+a_{2}=\ell$ and $a_{0}=a_{2}$ as sum of all in-degrees $a_{1}+2a_{2}$ must also be $\ell$. Thus, we can assume that $a_{0}=a_{2}=a$ and $a_{1}=\ell-2a$ for some $a\in \{1,2,\ldots,\lfloor \ell/2\rfloor\}$.  When $\delta =0$ or $\beta=1/2$, $F_{\ell}(x_{1},x_{2},\ldots,x_{\ell})\cong (\log n)^{\ell}$ for all $(\ell-1)!/2$ many embeddings.  Thus, assume that $\delta <0$ or $\beta<1/2$. In particular, for such an embedding the contribution to the mean is $F_{\ell}(x_{1},x_{2},\ldots,x_{\ell})$ where $a$ many $x_{i}$'s are $2\beta$, $a$ many $x_{i}$'s are $(2-2\beta)$ and $(\ell-2a)$ many $x_{i}$'s are $1$.  Given $a$, the maximum growth (in terms of $n$) occurs when $x_{1}=x_{2}=\cdots=x_{a}=2-2\beta$ with the value $n^{(1-2\beta)a}(\log n)^{\ell-2a}$.  Maximizing over $a$, we get the result. Note that, number of maximizing embedding is $\frac{1}{2}a!(a-1)!$ when $\ell=2a$, and $\frac{1}{2}a!(a+1)!$ when $\ell=2a+1$.
\end{proof}

We now conclude this section with a proof of the graphical computation lemma.
\begin{proof}[Proof of Lemma~\ref{lem:graphic}]
 The proof of the lemma consists of two parts. 
 First, we observe that given $i_1<i_2<\cdots<i_k$, we have
 \begin{align*}
 \E \prod_{(a,b)\in E\atop a>b}\theta_{i_b i_a}
 &= \E \prod_{j=1}^{k-1}\left(\psi_{i_j}^{d_j} \prod_{t=i_{j}+1}^{i_{j+1}-1} (1-\psi_t)^{\ell_j}\right)= \prod_{j=1}^{k-1}\frac{\E(\psi_{i_j}^{d_j})}{\E\left((1-\psi_{i_j})^{\ell_j}\right)}\cdot \prod_{j=1}^{k-1} \prod_{t=i_{j}}^{i_{j+1}-1} \E\left((1-\psi_t)^{\ell_j}\right).
 \end{align*}
 One can explicitly derive the moments to compute the sum, but for simplicity, we approximate them to determine their order. As before, we do not write the explicit constants appearing in the ratio, but use the $\cong$ symbol to mean that the ratio of two sides is uniformly bounded away from zero and infinity (the constant will depend on $H$ and $m,\delta$).

 Observe that
 \begin{align*}
 \abs{\E \psi_t - \beta\log(1+1/t)} 
 = \abs{\beta\left(t-\frac{2m}{2m+\delta}\right)^{-1} - \frac{\beta}{t} + \frac{\beta}{t} - \beta\log(1+1/t)}
 \le \frac{C_1}{t^2} 
 \end{align*}
 for all $t\ge 2$ and some constant $C_1$.
 Using Lemma~\ref{lem:powerk} and the fact that $\E \psi_t\le \E \psi_2$ for all $t\ge 2$, we get that
 \begin{align*}
 &\abs{\log\prod_{j=1}^{k-1}\prod_{t=i_{j}}^{i_{j+1}-1} \E(1-\psi_t)^{\ell_j}
 + \sum_{j=1}^{k-1} \sum_{t=i_{j}}^{i_{j+1}-1} \ell_j \E\psi_t }\\
 &\le \sum_{j=1}^{k-1} \sum_{t=i_{j}}^{i_{j+1}-1} \ell_j\cdot \max\left\{\frac{\ell_j-1}{2},\frac{1}{1-\E \psi_2}\right\} \cdot \E\psi_t^2 
 \le C_2\cdot \sum_{t=2}^\infty \E\psi_t^2 \le C_3<\infty
 \end{align*}
 and
 \begin{align*}
 \abs{\sum_{j=1}^{k-1} \sum_{t=i_{j}}^{i_{j+1}-1} \ell_j \E\psi_t - \sum_{j=1}^{k-1} \beta\ell_j \log(i_{j+1}/i_j)}\le C_4<\infty.
 \end{align*}
 Here, constants $C_i$, $i\in[4]$, depend only on $H,m,$ and $\delta$. 
 Therefore, we get
 \begin{align*}
 \E \prod_{(a,b)\in E \atop a>b}\theta_{i_b i_a} \cong \prod_{j=1}^{k-1}\frac{1}{i_j^{d_j}} \cdot \left(\frac{i_j}{i_{j+1}}\right)^{\beta \ell_j} = \prod_{j=1}^k i_j^{-a_j}
 \end{align*}
 where $a_j:=d_j - \beta(\ell_j-\ell_{j-1}), j\in[k]$, in the sense that the ratio is uniformly bounded away from zero and infinity.

 Now, we define the function
 \begin{align*}
 F_k(a_1,a_2,\ldots,a_k):=
 \sum_{1\le i_1< i_2< \cdots< i_k\le n} \frac{1}{
 \prod_{j=1}^k i_j^{a_j}}.
 \end{align*}
 Observe that this is a function of $n$, but we abuse the notation to omit the dependence on $n$.
 The rest of the proof is by induction on $k$. We can easily compute $F_1(a_1)$ as a function of $n$. Consider the general case of $k>1$. If $a_k<1$ or $t=k$, we clearly have
 \begin{align*}
 F_k(a_1,a_2,\ldots,a_k) \cong F_{k-1}(a_1,a_2,\ldots,a_{k-1})\cdot n^{1-a_k}.
 \end{align*}
 Similarly, for $a_k=1$, we have
 \begin{align*}
 F_k(a_1,a_2,\ldots,a_k) \cong F_{k-1}(a_1,a_2,\ldots,a_{k-1})\cdot \log n.
 \end{align*}
 Next, if $a_k>1$, we have
 \begin{align*}
 F_k(a_1,a_2,\ldots,a_k) \cong F_{k-1}(a_1,a_2,\ldots,a_{k-1}+a_k-1).
 \end{align*}
 Using induction, we complete the proof.
\end{proof} 

\section{Proof of Theorem~\ref{thm:variance_order}}\label{sec:fluct_order}
As we mentioned before, the proof of Theorem~\ref{thm:variance_order} proceeds by establishing Lemmas~\ref{lem: lin term}--~\ref{lem:Tn3}. We observe that to prove Lemmas~\ref{lem: lin term}--\ref{lem:Tn3} we use the graphical computational lemma (Lemma~\ref{lem:graphic}), while the proof of Lemma~\ref{lem: centering term} uses yet another idea, where we represent a certain product as a sum of martingale differences. The latter idea is crucial in the proof of parts~\ref{thm:main_delta_zero} and~\ref{thm:main_delta_neg} Theorem~\ref{thm:main}.

\subsection{Proof of Lemma~\ref{lem:Tn3}}
 The statement follows directly from Theorem~\ref{thm:ET}.\qed

\subsection{Proof of Lemma~\ref{lem:quad_term}}
We have
\begin{align*}
\E\abs{\sum_{i<j}\theta_{ij}\sum_{k\neq i,j}\sum_{\ell_1,\ell_2=1}^m \overline{\go}_{ik}^{[\ell_1]}\overline{\go}_{jk}^{[\ell_2]}}
 & \le \sum_{\ell_1,\ell_2=1}^m \E\E_{\mvpsi}\abs{\sum_{i<j, k\neq i,j} \theta_{ij}\overline{\go}_{ik}^{[\ell_1]}\overline{\go}_{jk}^{[\ell_2]}}.
\end{align*}
Now, since for fixed $j$ and $\ell$ the random variables $\go_{i,j}^{\ell}$ are negatively correlated we get
\begin{align*}
 \E_{\mvpsi} \left(\sum_{i<j, k\neq i,j} \theta_{ij}\overline{\go}_{ik}^{[\ell_1]}\overline{\go}_{jk}^{[\ell_2]}\right)^2
 &\le \sum_{i<j, k\neq i,j} \theta_{ij}^2\E_{\mvpsi}\left(\left(\overline{\go}_{ik}^{[\ell_1]}\right)^2\right)\E_{\mvpsi}\left(\left(\overline{\go}_{jk}^{[\ell_2]}\right)^2\right)
 \le \sum_{i<j, k\neq i,j} \theta_{ij}^2\theta_{ik}\theta_{jk}
\end{align*}
and therefore
\begin{align*}
 \E\abs{\sum_{i<j}\theta_{ij}\sum_{k\neq i,j}\sum_{\ell_1,\ell_2=1}^m \overline{\go}_{ik}^{[\ell_1]}\overline{\go}_{jk}^{[\ell_2]}}
 & \le m^2 \E\sqrt{\sum_{i<j, k\neq i,j} \theta_{ij}^2\theta_{ik}\theta_{jk}}
 \le m^2\sqrt{\E\sum_{i<j, k\neq i,j} \theta_{ij}^2\theta_{ik}\theta_{jk}}.
\end{align*}

Finally,
\begin{align}
 \E\sum_{i<j\atop k\neq i,j}\theta_{ij}^2\theta_{ik}\theta_{jk}&=\frac{1}{3}\E\sum_{i<j<k} \theta_{ij}\theta_{ik}\theta_{jk}(\theta_{ij}+\theta_{ik}+\theta_{jk})\notag\\
 &\cong\E\sum_{i<j<k} \theta_{ij}\theta_{ik}\theta_{jk}(\theta_{ij}+\theta_{jk})\quad [\text{because } \theta_{ik}<\theta_{jk}]\notag\\
 &=\E\sum_{i<j<k} \theta_{ij}^2\theta_{ik}\theta_{jk}+\E\sum_{i<j<k} \theta_{ij}\theta_{ik}\theta_{jk}^2\label{eq:quad-terms}.
\end{align}
We now apply the graphical computation lemma (Lemma~\ref{lem:graphic}) for each of the terms separately, noting that they correspond to the graphs presented in Figure~\ref{fig:Graphs_in_quad-term} on the left and the right, respectively.
\begin{figure}[htb]
\begin{center}
\begin{tikzpicture}[thick, scale=0.7] 
 \node[big dot,label=below:$i$] (i) at (-8,0) {};
 \node[big dot,label=below:$j$] (j) at (-5,0) {};
 \node[big dot,label=below:$k$] (k) at (-2,0) {};
 \draw [line width=0.5mm] (i) -- (k);

 \draw [<-,red] (i) to [bend left=50](j);
 \draw [<-,red] (i) to [bend left=30](j);
 \draw [<-,red] (j) to [bend left=40](k);
 \draw [<-,red] (i) to [bend left=50](k);

 \draw (-8,1) node[above] {$\ell_i=3$};
 \draw (-8,-1) node[below] {$d_i=3$};
 
 \draw[dashed] (-5,-0.2) rectangle (-4.5,1.7);
 \draw (-5,2) node[above] {$\ell_j=2$};
 \draw (-5,-1) node[below] {$d_j=1$};
 
 \draw (-2,-1) node[below] {$d_k=0$};
 \draw (-2,1) node[above] {$\ell_k=0$};

 \node[big dot,label=below:$i$] (i1) at (2,0) {};
 \node[big dot,label=below:$j$] (j1) at (5,0) {};
 \node[big dot,label=below:$k$] (k1) at (8,0) {};
 \draw [line width=0.5mm] (i1) -- (k1);

 \draw [<-,red] (j1) to [bend left=50](k1);
 \draw [<-,red] (j1) to [bend left=30](k1);
 \draw [<-,red] (i1) to [bend left=40](j1);
 \draw [<-,red] (i1) to [bend left=50](k1);

 \draw (2,1) node[above] {$\ell_i=2$};
 \draw (2,-1) node[below] {$d_i=2$};
 
 \draw[dashed] (5,-0.2) rectangle (5.5,1.7);
 \draw (5,2) node[above] {$\ell_j=3$};
 \draw (5,-1) node[below] {$d_j=2$};
 
 \draw (8,-1) node[below] {$d_k=0$};
 \draw (8,1) node[above] {$\ell_k=0$};
 \end{tikzpicture}	
\end{center}
\caption{Graphical computation diagrams for~\eqref{eq:quad-terms}.\label{fig:Graphs_in_quad-term}}
\end{figure}
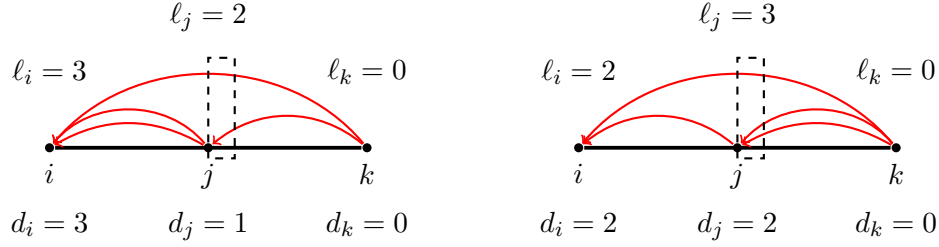

So, by Lemma~\ref{lem:graphic}, recalling that $\beta:=(m+\delta)/(2m+\delta)$, we get
\begin{align*}
 \E\sum_{i<j<k} \theta_{ij}^2\theta_{ik}\theta_{jk}\cong F_3(3-3\beta,1+\beta,2\beta)=\begin{cases}
 1 & \text{ if } \beta>\frac{1}{2}\\
 \log n & \text{ if } \beta=\frac12\\
 n^{1-2\beta} & \text{ if } \beta<\frac12,
 \end{cases}
\end{align*}
and
\begin{align*}
 \E\sum_{i<j<k} \theta_{ij}\theta_{ik}\theta_{jk}^2\cong F_3(2-2\beta,3-1\beta,3\beta)=\begin{cases}
 1 & \text{ if } \beta>\frac{1}{3} \\
 \log n & \text{ if } \beta=\frac13\\
 n^{1-3\beta} & \text{ if } \beta<\frac13
 \end{cases}.
\end{align*}

Comparing these values in respective intervals for $\beta$ yields that the former expectation always dominates the latter and concludes the proof. \qed

\subsection{Proof of Lemma~\ref{lem: lin term}}
 By similar computations to those in Lemma~\ref{lem:quad_term}, it is enough to upper bound the following expectation
\begin{align*}
 &\E\sum_{i<j}\theta_{ij}\left(\sum_{k\neq i,j}\theta_{ik}\theta_{jk}\right)^2\\
 &=\E\sum_{i<j}\theta_{ij}\left[\sum_{k=j+1}^n\theta_{ik}\theta_{jk}+\sum_{k=i+1}^{j-1}\theta_{ik}\theta_{kj}+\sum_{k=1}^{i-1}\theta_{ki}\theta_{kj}\right]^2\\
 &\le \underbrace{\E\sum_{i<j}\theta_{ij}\left[\sum_{k=j+1}^n\theta_{ik}\theta_{jk}\right]^2}_{A_1}+\underbrace{\E\sum_{i<j}\theta_{ij}\left[\sum_{k=i+1}^{j-1}\theta_{ik}\theta_{kj}\right]^2}_{A_2}+\underbrace{\E\sum_{i<j}\theta_{ij}\left[\sum_{k=1}^{i-1}\theta_{ki}\theta_{kj}\right]^2}_{A_3}.
\end{align*}

We now treat each of $A_1$, $A_2$, and $A_3$ separately using the graphical computation lemma (Lemma~\ref{lem:graphic}), noting that they correspond to the graphs presented in their respective figures. First for $A_1$ we get
\begin{align*}
 A_1&=\E\sum_{i<j}\theta_{ij}\sum_{i<j<k\le k'}\theta_{ik}\theta_{jk}\theta_{ik'}\theta_{jk'}
 \cong F_3(3-3\beta, 2-\beta,4\beta)+F_4(3-3\beta,2-\beta,2\beta,2\beta),
 \end{align*}
where
 \begin{align*}
F_3(3-3\beta, 2-\beta,4\beta) &\cong
 \begin{cases}
 1 & \text{ if } \beta>1/4\\
 \log n & \text{ if } \beta=1/4\\
 n^{1-4\beta} & \text{ if } \beta<1/4
 \end{cases}
\end{align*}
and
\begin{align*}
F_4(3-3\beta,2-\beta,2\beta,2\beta)
 &\cong\begin{cases}
 1 & \text{ if } \beta>1/2\\
 \log n & \text{ if } \beta=1/2\\
 n^{2-4\beta} & \text{ if } \beta<1/2
 \end{cases}.
\end{align*}

\begin{figure}[h!]
\begin{center}
\begin{tikzpicture}[thick, scale=0.5] 
 \node[big dot,label=below:$i$] (i3) at (-11,0) {};
 \node[big dot,label=below:$j$] (j3) at (-8,0) {};
 \node[big dot,label=below:${k=k'}$] (k3) at (-5,0) {};
 \draw [line width=0.5mm] (i3) -- (k3);
    
 \draw [<-,red] (i3) to [bend left=40](j3);
 \draw [<-,red] (j3) to [bend left=50](k3);
 \draw [<-,red] (j3) to [bend left=30](k3);
 \draw [<-,red] (i3) to [bend left=45](k3);
 \draw [<-,red] (i3) to [bend left=65](k3);
\draw (-11,1.5) node[above] {$\ell_i=3$};
 \draw (-11,-1.2) node[below] {$d_i=3$};
 
 \draw[dashed] (-8,-0.2) rectangle (-7.5,2);
 \draw (-8,2) node[above] {$\ell_j=4$};
 \draw (-8,-1.2) node[below] {$d_j=2$};
 
 \draw (-5,1.5) node[above] {$\ell_k=4$};
 \draw (-5,-1.2) node[below] {$d_k=0$};

 \node[big dot,label=below:$i$] (i) at (0,0) {};
 \node[big dot,label=below:$j$] (j) at (3,0) {};
 \node[big dot,label=below:$k$] (k) at (6,0) {};
 \node[big dot,label=below:$k'$] (k1) at (9,0) {};
 \draw [line width=0.5mm] (i) -- (k1);

 \draw [<-,red] (i) to [bend left=40](j);
 \draw [<-,red] (j) to [bend left=40](k);
 \draw [<-,red] (j) to [bend left=45](k1);
 \draw [<-,red] (i) to [bend left=40](k);
 \draw [<-,red] (i) to [bend left=40](k1);
 
 \draw (0,1) node[above] {$\ell_i=3$};
 \draw (0,-1.2) node[below] {$d_i=3$};
 
 \draw[dashed] (3,-0.2) rectangle (3.5,2);
 \draw (3,2) node[above] {$\ell_j=4$};
 \draw (3,-1.2) node[below] {$d_j=2$};

 \draw[dashed] (6,-0.2) rectangle (6.5,1.8);
 \draw (6,2) node[above] {$\ell_k=2$};
 \draw (6,-1.2) node[below] {$d_k=0$};
 
 \draw (9,-1.2) node[below] {$d_{k'}=0$};
 \draw (9,1) node[above] {$\ell_{k'}=0$};
\end{tikzpicture}	
\end{center}
\caption{Graphical computation diagrams for $A_1$.\label{fig:Graphs_in_lin-term-A1}}
\end{figure}
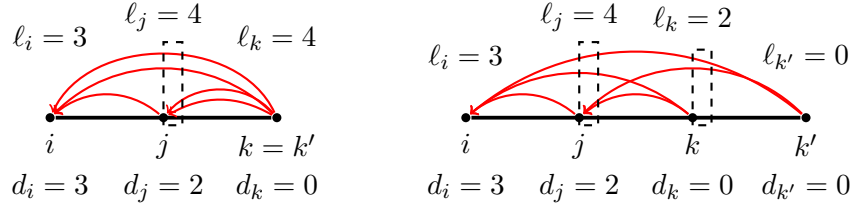

Term $A_2$ is computed similarly 
\begin{align*}
 A_2=\E\sum_{i<j}\theta_{ij}\sum_{i<k\le k'<j}\theta_{ik}\theta_{jk}\theta_{ik'}\theta_{jk'}=F_3(3-3\beta,2,3\beta)+F_4(3-3\beta, 1,1,3\beta),
\end{align*}
where
\begin{align*}
 F_3(3-3\beta,2,3\beta)\cong \begin{cases}
 1 & \text{ if } \beta>1/3\\
 \log n & \text{ if } \beta=1/3\\
 n^{1-3\beta} & \text{ if } \beta< 1/3
 \end{cases}
\end{align*}
and
\begin{align*}
 F_4(3-3\beta, 1,1,3\beta)&\cong \begin{cases}
 1 & \text{ if } \beta>1/3\\
 (\log n)^3 & \text{ if } \beta=1/3\\
 n^{1-3\beta}(\log n)^2 & \text{ if } \beta< 1/3.
 \end{cases}
\end{align*}

\begin{figure}[htb]
\begin{center}
\begin{tikzpicture}[thick, scale=0.5] 
  \node[big dot,label=below:$i$] (i3) at (-11,0) {};
 \node[big dot,label=below:$j$] (j3) at (-5,0) {};
 \node[big dot,label=below:${k=k'}$] (k3) at (-8,0) {};
 \draw [line width=0.5mm] (i3) -- (j3);
    
 \draw [<-,red] (i3) to [bend left=60](j3);
 \draw [<-,red] (i3) to [bend left=30](k3);
 \draw [<-,red] (i3) to [bend left=50](k3);
 \draw [<-,red] (k3) to [bend left=30](j3);
 \draw [<-,red] (k3) to [bend left=50](j3);

  \draw (-11,1.5) node[above] {$\ell_i=3$};
 \draw (-11,-1.2) node[below] {$d_i=3$};
 
 \draw[dashed] (-8,-0.2) rectangle (-7.5,2);
 \draw (-5,1.5) node[above] {$\ell_j=0$};
 \draw (-5,-1.2) node[below] {$d_j=0$};
 
 \draw (-8,2) node[above] {$\ell_k=3$};
 \draw (-8,-1.2) node[below] {$d_k=2$};

 \node[big dot,label=below:$i$] (i) at (0,0) {};
 \node[big dot,label=below:$k$] (k) at (3,0) {};
 \node[big dot,label=below:$k'$] (k1) at (6,0) {};
 \node[big dot,label=below:$j$] (j) at (9,0) {};
 \draw [line width=0.5mm] (i) -- (j);

 \draw [<-,red] (i) to [bend left=40](j);
 \draw [->,red] (j) to [bend right=40](k);
 \draw [->,red] (j) to [bend right=45](k1);
 \draw [<-,red] (i) to [bend left=40](k);
 \draw [<-,red] (i) to [bend left=40](k1);
 
 \draw (0,1) node[above] {$\ell_i=3$};
 \draw (0,-1) node[below] {$d_i=3$};
 
 \draw[dashed] (3,-0.2) rectangle (3.5,2);
 \draw (3,2) node[above] {$\ell_k=3$};
 \draw (3,-1) node[below] {$d_k=1$};

 \draw[dashed] (6,-0.2) rectangle (6.5,1.8);
 \draw (6,2) node[above] {$\ell_{k'}=3$};
 \draw (6,-1) node[below] {$d_{k'}=1$};
 
 \draw (9,-1) node[below] {$d_{j}=0$};
 \draw (9,1) node[above] {$\ell_{j}=0$};
\end{tikzpicture}	
\end{center}
\caption{Graphical computation diagrams for $A_2$.\label{fig:Graphs_in_lin-term-A2}}
\end{figure}
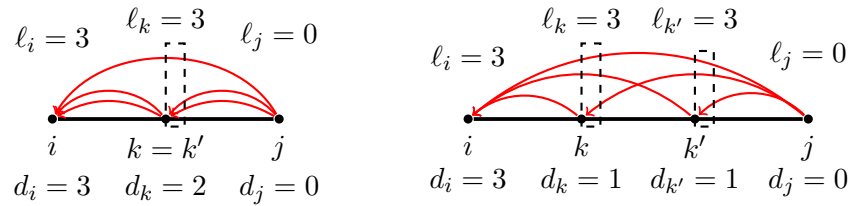

Finally, we compute the last term $A_3$.
\begin{align*}
 A_3&=\E\sum_{i<j}\theta_{ij}\sum_{k\le k'<i}\theta_{ik}\theta_{jk}\theta_{ik'}\theta_{jk'}\\
 &=F_3(4-4\beta,1+\beta,3\beta)+F_4(2-2\beta,2-\beta,1+\beta,3\beta),
 \end{align*}
 where
 \begin{align*}
     F_3(4-4\beta,1+\beta,3\beta)\cong \begin{cases}
 1 & \text{ if } \beta>1/3\\
 \log n & \text{ if } \beta=1/3\\
 n^{1-3\beta} & \text{ if } \beta< 1/3
 \end{cases}
 \end{align*} 
 and
 \begin{align*}
 F_4(2-2\beta,2-\beta,1+\beta,3\beta)\cong \begin{cases}
  1 & \text{ if } \beta>1/3\\
 \log n & \text{ if } \beta=1/3\\
 n^{1-3\beta} & \text{ if } \beta< 1/3
 \end{cases}.
\end{align*}

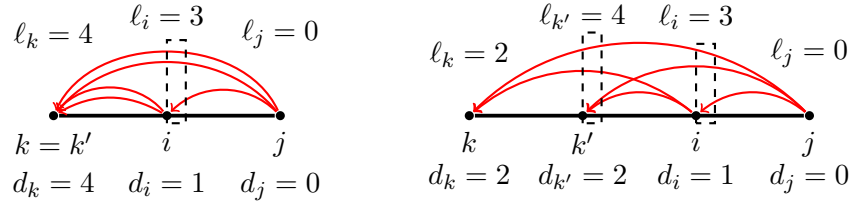
\begin{figure}[htb]
\begin{center}
\begin{tikzpicture}[thick, scale=0.5] 
\node[big dot,label=below:$i$] (i3) at (-8,0) {};
 \node[big dot,label=below:$j$] (j3) at (-5,0) {};
 \node[big dot,label=below:${k=k'}$] (k3) at (-11,0) {};
 \draw [line width=0.5mm] (k3) -- (j3);
    
 \draw [<-,red] (k3) to [bend left=30](i3);
 \draw [<-,red] (k3) to [bend left=50](i3);
 \draw [<-,red] (i3) to [bend left=45](j3);
 \draw [<-,red] (k3) to [bend left=50](j3);
 \draw [<-,red] (k3) to [bend left=65](j3);

  \draw (-8,2) node[above] {$\ell_i=3$};
 \draw (-8,-1.2) node[below] {$d_i=1$};
 
 \draw[dashed] (-8,-0.2) rectangle (-7.5,2);
 \draw (-5,1.5) node[above] {$\ell_j=0$};
 \draw (-5,-1.2) node[below] {$d_j=0$};
 
 \draw (-11,1.5) node[above] {$\ell_k=4$};
 \draw (-11,-1.2) node[below] {$d_k=4$};

 \node[big dot,label=below:$k$] (k) at (0,0) {};
 \node[big dot,label=below:$k'$] (k1) at (3,0) {};
 \node[big dot,label=below:$i$] (i) at (6,0) {};
 \node[big dot,label=below:$j$] (j) at (9,0) {};
 \draw [line width=0.5mm] (k) -- (j);

 \draw [<-,red] (i) to [bend left=40](j);
 \draw [->,red] (j) to [bend right=45](k);
 \draw [->,red] (j) to [bend right=45](k1);
 \draw [->,red] (i) to [bend right=40](k);
 \draw [->,red] (i) to [bend right=40](k1);
 
 \draw (0,1) node[above] {$\ell_k=2$};
 \draw (0,-1) node[below] {$d_k=2$};
 
 \draw[dashed] (3,-0.2) rectangle (3.5,2.2);
 \draw (3,2) node[above] {$\ell_{k'}=4$};
 \draw (3,-1) node[below] {$d_{k'}=2$};

 \draw[dashed] (6,-0.2) rectangle (6.5,1.9);
 \draw (6,2) node[above] {$\ell_{i}=3$};
 \draw (6,-1) node[below] {$d_{i}=1$};
 
 \draw (9,-1) node[below] {$d_{j}=0$};
 \draw (9,1) node[above] {$\ell_{j}=0$};
\end{tikzpicture}	
\end{center}
\caption{Graphical computation diagrams for $A_3$.\label{fig:Graphs_in_lin-term-A3}}
\end{figure}

Comparing the orders of the terms $A_1$, $A_2$, and $A_3$ and recalling that $\beta:=(m+\delta)/(2m+\delta)$ concludes the proof.\qed

\subsection{Proof of Lemma~\ref{lem: centering term}}
Recall definitions of the quantity $\mu_{i,j,k}$ and the normalized random variables $X_t,Y_t,U_t$ from~\eqref{eq:mu_ijk} and~\eqref{eq:XYU}, respectively.
Rewriting the product of $ Y^{[j]}_t$ as the sum of martingale differences we get
\begin{align*}
    \prod_{t=i+1}^{k-1}Y^{[j]}_t=1+\sum_{\ell=i+1}^{k-1} Y^{[j]}_{i+1}\cdots Y^{[j]}_{\ell-1} \left( Y^{[j]}_{\ell}-1\right).
\end{align*}
 
Using the definition of $\theta_{ij}$ we can now rewrite term~\eqref{term:central0}, as follows
\begin{align}
    \Delta_n
    &=\sum_{i<j<k}\theta_{ij}\theta_{ik}\theta_{jk}\notag\\
    &=\sum_{i=1}^{n-2} \psi_i^2\cdot \sum_{k=i+1}^{n} \left[\left(\prod_{t=i+1}^{k-1}(1-\psi_t)^2\right)\cdot \left(\sum_{j=i+1}^{k-1}\frac{\psi_j}{1-\psi_j}\right)\right]\notag\\
    &=\sum_{i=1}^{n-2} \psi_i^2\sum_{i<j<k\le n}\left[ \left(\prod_{t=i+1}^{j-1}(1-\psi_t)^2\right)\cdot \psi_j(1-\psi_j)\left(\prod_{t=j+1}^{k-1}(1-\psi_t)^2\right)\right]\notag\\
    &=\sum_{i=1}^{n-2} \sum_{i<j<k\le n}\mu_{i,j,k} X_i\prod_{t=i+1}^{k-1}Y^{[j]}_t.\notag
\end{align}
Rewriting $\prod_{t=i+1}^{k-1}Y^{[j]}_t$ as a telescoping sum of martingale differences we get
\begin{align}
    \Delta_n&=\sum_{1\le i<j<k\le n} \mu_{i,j,k} X_i\left(1+\sum_{\ell=i+1}^{k-1} Y^{[j]}_{i+1}\cdots Y^{[j]}_{\ell-1} \left( Y^{[j]}_{\ell}-1\right)\right)\notag\\
    &=\sum_{1\le i<j<k\le n} \mu_{i,j,k}X_i \label{eq:central0-2_const}\\
    &\quad+\sum_{1\le i<j<k\le n} \mu_{i,j,k}X_i\left(\sum_{\ell=i+1}^{k-1} Y^{[j]}_{i+1}\cdots Y^{[j]}_{\ell-1} \left( Y^{[j]}_{\ell}-1\right)\right).\label{eq:central0-2}
\end{align}
Now assume that $\delta> 0$ ($\beta>1/2$). Recalling that $\sup_{i\ge 1}\var(X_i)\cong1$ and $\mu_{i,\cdot,\cdot}^{(n)} \cong i^{-1}$ for $i\ge 1$ from Lemma~\ref{lem:mu_ijk}, we get
\begin{align}
    \var\left(\eqref{eq:central0-2_const}\right)
    =\sum_{i=1}^n \var(X_i)\cdot  \left\vert\mu_{i,\cdot,\cdot}^{(n)}\right\vert^2<\infty. 
\end{align}
and
\begin{align*}
 \var\left(\eqref{eq:central0-2}\right)
 &\cong\sum_{{i<j<k\atop i'<j'<k'}\atop i<i'} \frac{1}{(ii')^{2-2\beta}(kk')^{2\beta}jj'}\cdot \E\left[ \prod_{t=i+1}^{i'}Y_t^{[j]}\cdot\sum_{\ell=i'}^{k\wedge k'}\left(\prod_{t=i'+1}^{\ell-1} Y_{t}^{[j]}Y_{t}^{[j']}\right)\cdot \left(Y^{[j]}_{\ell}-1\right)\left(Y^{[j']}_{\ell}-1\right)\right]\\
 &\cong\sum_{{i<j<k\atop i'<j'<k'}\atop i<i'} \frac{1}{(ii')^{2-2\beta}(kk')^{2\beta}jj'}\cdot \frac{1}{i'}\\
 &\cong\sum_{i<i' \atop{i<j,\ i'<j'}} \frac{1}{i^{2-2\beta}(i')^{3-2\beta}(jj')^{2\beta}}
 \cong\sum_{i<i'} \frac{1}{i^{2-2\beta}(i')^{3-2\beta}(ii')^{2\beta}}<\infty.
\end{align*}

The case when $\delta=0$ ($\beta=1/2$) is treated similarly,
\begin{align*}
\var\left(\eqref{eq:central0-2_const}\right)
 &=\sum_{i=1}^n \var(X_i)\cdot \mu_{i,\cdot,\cdot}^2=\sum_{i=1}^n \frac{(\log n)^4}{i^2}\cong (\log n)^4.
\end{align*}
and
\begin{align*}
 \var\left(\eqref{eq:central0-2}\right)&\cong \sum_{{i<j<k\atop i'<j'<k'}\atop i<i'} \frac{1}{ii'^2jj'kk'}\cong (\log n)^4.
\end{align*}

Finally, again by a similar computation when $\delta\in(-m,0)$ ($\beta<1/2$) we get
\begin{align*}
 \var\left(\eqref{eq:central0-2_const}\right)
 &=\sum_{i=1}^n \var(X_i)\cdot \mu_{i,\cdot,\cdot}^2
 \cong \sum_{i=1}^n 1\cdot n^{2-4\beta}(\log n)^2 i^{4\beta-4}\cong n^{2-4\beta}(\log n)^2.
\end{align*}
and
\begin{align*}
 \var\left(\eqref{eq:central0-2}\right)&\cong \sum_{{i<j<k\atop i'<j'<k'}\atop i<i'} \frac{1}{(ii')^{2-4\beta}(kk')^{2\beta}jj'}\cdot \frac{1}{i'}
 \cong\sum_{i<i' \atop{i<j,\ i'<j'}} \frac{n^{2-4\beta}}{i^{2-2\beta}(i')^{3-2\beta}(jj')^{1}}\cong n^{2-4\beta}(\log n)^2.
\end{align*}
Combining, we have the result.\qed

\section{Distributional convergence for positive $\delta$}\label{sec:clt}
The goal of this section is to prove part~\ref{thm:main_delta_pos} of Theorem~\ref{thm:main}.
First, we recall Stein's method for exchangeable pairs.
This method is typically applicable to systems in which small perturbations do not significantly change the distribution. In particular, we do not know of any application of it to a non-static random graph. The following presentation is a modification of the classical results due to Stein~\cite{Stein72, Stein86} as in~\cite[Theorem 1.1]{DT_CCLT}.
\begin{thm}\label{th:st_exch}
	Let $(W, W')$ be an exchangeable pair of random variables defined on the same probability space. Suppose $\E W=0$, $\E W^2=1$, $\E \abs{W}^3<\infty $, and 
	\begin{align*}
		\E (W'-W\mid W)=-\lambda(W+R_1) \text{ and } \E \left((W'-W)^2\mid W\right)=2\lambda(1+R_2)
	\end{align*}
	almost surely for some constant $\lambda\in(0, 1)$ and random variables $R_i=R_i(W)$ for $i=1, 2$. Then
	\begin{align*}
		\dwas(W, Z) \le \E|R_1|+\sqrt{\frac{2}{\pi}}\E |R_2|+\frac{1}{3\lambda}\E|W'-W|^3,
	\end{align*}
	where $Z$ is a standard normal random variable and $\dwas(W, Z)$ denotes the Wasserstein-$1$ distance.
\end{thm}

We are now ready to present 

\begin{proof}[Proof of Theorem~\ref{thm:main} part~\ref{thm:main_delta_pos}]

Recall the centered-edge decomposition of $T_n$ as in~\eqref{term:central0}--\eqref{term:central3}, by Theorem~\ref{thm:variance_order} we know that when $\delta>0$ dominating term is of fluctuation order $\sqrt{\log n}$ is given by $T_{n}^{(3)}$, as in~\eqref{term:central3}, while all other terms have constant order fluctuation. Hence, setting 
\begin{align*}
W_n=T_{n}^{(3)}=\sum_{1\le i<j<k\le n}\sum_{\ell_1,\ell_2,\ell_3=1}^m\overline{\go}_{ij}^{[\ell_1]}\overline{\go}_{ik}^{[\ell_2]}\overline{\go}_{jk}^{[\ell_3]},
\end{align*}
we will first show the conditional version of the result for $(W_n\mid \mvpsi_n)$.
First the conditional variance of $W_n$ is given by
\begin{align*}
    \gs^2_n(\mvpsi):=\var_{\mvpsi} (W_n)&=m^2(m-1)\sum_{1\le i<j<k\le n}\theta_{ij}\theta_{jk}\theta_{ik}(1-\theta_{ij})(1-\theta_{ik})(1-\theta_{jk}).
\end{align*}
At the end of the proof, we will remove the conditioning using the fact that
\begin{align}\label{gs_vs_nu}
    \E \abs{\Delta_{n}-\gs^2_n(\mvpsi)}\cong 1,
\end{align}
given to us by Lemma~\ref{cor:cycles_order}.\ref{cor:cycles_order_part2}.
 
The proof proceeds by the exchangeable pair approach of Stein's method as in Theorem~\ref{th:st_exch}. Recall that by the construction of the P\'olya urn graph in Subsection~\ref{sec:model} the random variables $\{U_{j,\ell}\}_{j\in[n],\ell\in[m]}$ are independent conditioned on the environment $\mvpsi_n$. Hence, to create an exchangeable pair, it is natural to resample one of them chosen uniformly at random. Let $J\sim\mathrm{Uniform[n]}$ be a uniformly chosen vertex and $L\sim\mathrm{Uniform[m]}$ be a uniformly chosen outgoing edge from $J$. Recall that 
    \[
    \go_{ij}^{[\ell]}=\ind_{U_{j,\ell}\in I_i^{(n)}}=\ind_{\{j\to i\}}(\ell)
    \] 
is the indicator function that $j$-th vertex connected to $i$-th vertex with the $\ell$-th edge. Given $\mvpsi_n$, let 
\begin{align*}
    &\{U'_{j,\ell}\}_{j\in[n],\ell\in[m]}\text{  be independent copies of } \{U_{j,\ell}\}_{j\in[n],\ell\in[m]},\\
    &(\go_{ij}^{[\ell]})'=\ind_{U'_{j,\ell}\in I_i^{(n)}}, \text{ and}\\
    & W_n' \text{ the version of } W_n,\text{ where } U_{J,L} \text{ is replaced by its independent copy } U_{J,L}'.
\end{align*}
Finally, let $\cT_{J, L}$ denote the set of triangles that contain the $L$-th outgoing edge from the $J$-th vertex, and $\cF_n$ be the natural filtration of the original model. We have
 \begin{align}
 W_n'-W_n
 =\sum_{\cT_{J,L}}\left((\overline{\go}_{iJ}^{[L]})'-\overline{\go}_{iJ}^{[L]}\right)\cdot \overline{\go}_{ik}^{[\ell_1]}\overline{\go}_{Jk}^{[\ell_2]}.\label{eq:perturb}
 \end{align} 
 In particular, the linearity condition (with $R_1=0$) holds from Theorem~\ref{th:st_exch} 
 \begin{align*}
 \E_{\mvpsi} (W_n'-W_n\mid \cF_n)
 &=-\frac{1}{nm}\sum_{j,\ell}\left(\sum_{\cT_{j,\ell}}\left(\E\left((\overline{\go}_{ij}^{[\ell]})'\mid\cF_n\right)-\overline{\go}_{ij}^{[\ell]}\right)\overline{\go}_{ik}^{[\ell_1]}\overline{\go}_{jk}^{[\ell_2]}\right)=-\frac{3}{nm}W_n,
 \end{align*} 
where we used that $\{(\overline{\go}_{ij}^{[\ell]})'\}$ are independent copies of $\{\overline{\go}_{ij}^{[\ell]}\}$.
 
Denoting $\lambda:=3/nm$, next we show that $\E_{\mvpsi} ((W_n'-W_n)^2\mid \cF_n)\cong 2\lambda$. Since 
\begin{align*}
    \E_{\mvpsi}\left(\left((\overline{\go}_{ij}^{[\ell]})'-\overline{\go}_{ij}^{[\ell]}\right)^2 \mid \cF_n\right)
    &=(1-2\theta_{ij})\cdot\overline{\go}_{ij}^{[\ell]}+2\theta_{ij}(1-\theta_{ij})\\
    \text{ and }\qquad\left(\overline{\go}_{ik}^{[\ell]}\right)^2 &= (1-\theta_{ik})\cdot \overline{\go}_{ik}^{[\ell]} + \theta_{ik}(1-\theta_{ik})
\end{align*} 
using equation~\eqref{eq:perturb}, we get
\begin{align*}
 \E_{\mvpsi} \left((W_n'-W_n)^2\mid \cF_n\right)
&=2\lambda \left(\gs_n(\mvpsi)^2+\overline{r}_1(\mvpsi)\right).
\end{align*}
We now claim that $\E\abs{\overline{r}_1(\mvpsi)}=O(1)$. Indeed, using representation of $W'-W$ from~\eqref{eq:perturb}, we can explicitly write out $\E_{\mvpsi} ((W_n'-W_n)^2\mid \cF_n)$ and observe that all of the terms besides 
$$\theta_{ij}\theta_{jk}\theta_{ik}(1-\theta_{ij})(1-\theta_{ik})(1-\theta_{jk})$$ 
corresponds to a graph that is the intersection of two triangles. 
Using the same Cauchy--Schwarz trick as in the proof of Lemma~\ref{lem:quad_term} to go from expressions involving $\overline{\go}$ to the ones purely in terms of $\theta$'s, we now may apply Lemma~\ref{cor:cycles_order}.\ref{cor:cycles_order_part2} and conclude that such terms are of constant order.

 Finally, to upper bound the cubic error term from Theorem~\ref{th:st_exch}, we claim something stronger, namely that almost surely (given $\mvpsi$)
 \begin{align*}
 \E_{\mvpsi} \abs{W_n'-W_n}^3\le \lambda(\gs^2_n(\mvpsi) +r_2(\mvpsi)), 
 \end{align*}
 where $r_2(\mvpsi)$ is a random variable with a universally bounded mean.
 
 To show that, first observe that since $\{\go_{i,j}\}_{i,j\in [n]}$ are $\{0,1\}$-valued random variables we have that 
 $\E_{\mvpsi}\abs{\overline{\go}_{i,j}}^3\le \theta_{ij}(1-\theta_{ij})$.
 Therefore, we derive
 \begin{align*}
 \lambda^{-1}(\log n)^{\frac{3}{2}}\E_{\mvpsi} \abs{W_n'-W_n}^3&\lesssim \sum_{j,\ell}\E_{\mvpsi}\left(\sum_{\cT_{j,\ell}}\abs{\overline{\go}^{[\ell]}_{ij}}\abs{\overline{\go}^{[\ell_1]}_{jk}}\abs{\overline{\go}^{[\ell_2]}_{ik}}\right)^3\\
 &\lesssim \E_{\mvpsi}\sum_{\cT}\abs{\overline{\go}^{[\ell]}_{ij}}\abs{\overline{\go}^{[\ell_1]}_{jk}}\abs{\overline{\go}^{[\ell_2]}_{ik}}+\sum_{\cH}\prod_{(x,y)\in H}\abs{\overline{\go}^{[\ell]}_{x,y}}\\
 &\lesssim \sum_{i<j<k}\theta_{ij}\theta_{ik}\theta_{jk}+\sum_{\cH}\prod_{(x,y)\in H}\theta_{x,y}\\
 &\lesssim \gs^2_n(\mvpsi)+r_2(\mvpsi),
 \end{align*}
 where $\cH$ is the set of all multigraphs that are the union of $3$ triangles, all of which have at least one common vertex. Therefore, by Lemma~\ref{cor:cycles_order}, $\E r_2(\mvpsi)$ is bounded by a constant that depends only on parameters $m$ and $\delta$.
 
Applying Theorem~\ref{th:st_exch} we get the Gaussian approximation for $W_n$ conditioned on $\mvpsi$
\begin{align}\label{eq:CCLT}
 \E\dwas\left(\left(\frac{W_n}{\sqrt{\log n}} \;\bigl|\; \mvpsi \right), \frac{\gs_n(\mvpsi)}{\sqrt{\log n}}\cdot Z\right)\lesssim \frac{1}{\sqrt{\log n}}.
\end{align}

Finally, recall that, $\overline{T}_n=(\nu_n(\mvpsi)^2-\E T_n) + \sum_{i=1}^3 T_{n,i}$ where 
 \[
 \nu_n(\mvpsi)^2:=m^2(m-1)\cdot \Delta_n.
 \]
Therefore, setting $\gs^2:=\lim_{n\to\infty}\E \gs^2_n(\mvpsi)/\log n$ we derive
 \begin{align*}
 &\dwas\left(\frac{\overline{T}_n}{\sqrt{\log n}}, \gs\cdot Z\right)\le \frac{\E \abs{\nu_n(\mvpsi)^2-\E T_n} + \E\abs{T_{n}^{(1)}}+\E\abs{T_{n}^{(2)}}}{\sqrt{\log n}}
 + \dwas\left(\frac{W_n}{\sqrt{\log n}}, \gs\cdot Z\right).
 \end{align*}
Now, for $\delta>0$, by Theorem~\ref{thm:variance_order} the numerator $\E \abs{\nu_n(\mvpsi)^2-\E T_n} + \E\abs{T_{n}^{(1)}}+\E\abs{T_{n}^{(2)}}$ is bounded by a constant. Therefore,
\begin{align*}
 \dwas\left(\frac{W_n}{\sqrt{\log n}}, \gs\cdot Z\right)
 &\lesssim \E\dwas\left(\left(\frac{W_n}{\sqrt{\log n}} \;\bigl|\; \mvpsi \right), \frac{\gs_n(\mvpsi)}{\log n}\cdot Z\right) + \E\dwas\left(\frac{\gs_n(\mvpsi)}{\sqrt{\log n}}\cdot Z, \gs\cdot Z\right).
\end{align*}
Moreover, one can easily check that $\dwas\left(\gs'\cdot Z, \gs\cdot Z\right) \le \sqrt{\pi/2}\cdot \abs{\gs-\gs'} $. Finally, we use Lemma~\ref{lem: centering term} to get that
\begin{align*}
 \E\dwas\left(\frac{\gs_n(\mvpsi)}{\sqrt{\log n}}\cdot Z, \gs\cdot Z\right) 
 &\le \sqrt{\frac{\pi}{2}}\cdot\E \abs{\frac{\gs_n(\mvpsi)}{\sqrt{\log n}}-\nu_n(\mvpsi)} \\
 &\le \frac{1}{\gs}\sqrt{\frac{\pi}{2}}\cdot\E \abs{\frac{\gs_n(\mvpsi)^2}{\log n}-\nu^2_n(\mvpsi)}
 \le \sqrt{\pi/2\cdot \var(\gs_n^2(\mvpsi)/\log n)}
 \lesssim \frac{1}{\sqrt{\log n}},
\end{align*}
where we applied~\eqref{gs_vs_nu}. This together with~\eqref{eq:CCLT} completes the proof.
\end{proof}

\section{Distributional convergence for non-positive $\delta$}\label{sec:delta_neg}
The goal of this section is to prove parts~\ref{thm:main_delta_zero} and~\ref{thm:main_delta_neg} of Theorem~\ref{thm:main}. Recall that
$
\beta=\frac{m+\delta}{2m+\delta}
$
and $\E(1-\psi_n)^2 = \Delta f_{n+1}/\Delta f_n$ for $n\ge 2$, where 
\begin{align*}
f_n=1+\sum_{t=3}^{n} \prod_{i=2}^{t-1} \E(1-\psi_i)^2, \quad n\ge 2
\end{align*}
and $f_1=0$. 
Also recall 
\begin{align*}
\tilde{s}_n&:= f_n\sum_{j=2}^{n-1} \frac{\E \psi_j(1-\psi_j)}{\E(1-\psi_j)^2}
- \sum_{j=2}^{n-1} \frac{\E \psi_j(1-\psi_j)}{\E(1-\psi_j)^2}\, f_j \notag\\
&= \bigl(1 +O(1/\log n)\bigr)\cdot
\begin{cases}
 \dfrac{\beta}{1-2\beta}n^{1-2\beta}\log n, & \text{if } 0<\beta<1/2,\\[2mm]
 \dfrac14(\log n)^2, & \text{if } \beta=1/2.
\end{cases}
\end{align*}
In particular, we have $\abs{\tilde{s}_n/s_{n}-1}\lesssim 1/\log n$.
We start by noting the following corollary of Lemma~\ref{lem:mu_ijk}.
\begin{cor}\label{lem:mu_ijk2}
For $m\ge 2$ and $\delta\in(-m,0]$, and for all $\ell\le n$,
\begin{align*}
\sum_{j<k\le n,\, j\le \ell} \mu_{i,j,k} \cong
\begin{cases}
\dfrac{\log \ell}{i^{2-2\beta}}\, n^{1-2\beta}, & \text{ if } \beta\in (0,1/2),\\[2mm]
\dfrac{\log \ell}{i^{2-2\beta}}\, \log n, & \text{ if } \beta=1/2,
\end{cases}
\end{align*}
and, in particular,
\begin{align*}
\frac{1}{\tilde{s}_n}\sum_{i=1}^{n-2}\mu_{i,\cdot,\cdot}^{(n)}\cong\begin{cases}
\log n & \text{ if } \beta=1/2,\\
1 & \text{ if } \beta\in (0,1/2).
\end{cases}
\end{align*}
\end{cor}
\begin{proof}
For fixed $i$,
\begin{align*}
\frac{\Delta f_{i+1}}{\E \psi_i^2}\sum_{i<j<k\le n\atop j\le \ell}\mu_{i,j,k}
=\sum_{j=i+1}^{\ell} \frac{\E \psi_j(1-\psi_j)}{\E(1-\psi_j)^2}(f_n-f_j).
\end{align*}
The stated asymptotics follow from the same estimates as in Lemma~\ref{lem:mu_ijk}, together with $\E\psi_i^2/\Delta f_{i+1}\cong i^{2\beta-2}$ and $\sum_{j\le \ell}j^{-1}\cong \log \ell$. The final claim follows by summing over $i$ and using~\eqref{eq:s_n}.
\end{proof}

\begin{proof}[Proof of Theorem~\ref{thm:main} parts~\ref{thm:main_delta_zero} and~\ref{thm:main_delta_neg}]

By Theorem~\ref{thm:variance_order}, the dominating term in the decomposition of
$
T_n=\eqref{term:central0}+\eqref{term:central1}+\eqref{term:central2}+\eqref{term:central3}
$
is the conditional-mean term~\eqref{term:central0}. Recall the representations of $\Delta_{n,i}$ and its mean $\E \Delta_{n,i}= \mu_{i,\cdot,\cdot}^{(n)}$ as given in~\eqref{eq:Triangle_i} and~\eqref{eq:mu_idotdot}.

For any $i\ge 1$, similar to a step in the proof of Lemma~\ref{lem: centering term}, we decompose $\Delta_{n,i}$ into the terms~\eqref{eq:central0-2_const} and~\eqref{eq:central0-2},
\begin{align*}
 \Delta_{n,i}=\underbrace{ \mu_{i,\cdot,\cdot}^{(n)} \cdot X_i}_{\eqref{eq:central0-2_const}}+\underbrace{\sum_{i<j<k\le n} \mu_{i,j,k}\left(\sum_{\ell=i+1}^{k-1} X_iY^{[j]}_{i+1}\cdots Y^{[j]}_{\ell-1} \left( Y^{[j]}_{\ell}-1\right)\right)}_{\eqref{eq:central0-2}}.
\end{align*}
We now treat the second term, namely~\eqref{eq:central0-2}. The goal is to sum the coefficients $\mu_{i,j,k}$ into $\mu_{i,\cdot,\cdot}^{(n)}$. To do this, we (i) change the order of summation and (ii) pass from terms involving $Y_t^{[j]}$ to $Y_t$ (recall $Y_t^{[j]}=Y_t$ whenever $t<j$). We do this in a sequence of steps, collecting the leftover terms into the error terms $r_{i,1}$ and $r_{i,2}$:
\begin{align}
~\eqref{eq:central0-2}
 &=X_i \cdot\sum_{i<j<k\le n} \mu_{i,j,k}\left(\sum_{\ell=i+1}^{k-1} Y^{[j]}_{i+1}\cdots Y^{[j]}_{\ell-1} \left( Y^{[j]}_{\ell}-1\right)\right)\notag\\
 &=X_i \cdot\sum_{\ell=i+1}^{n-1} \sum_{j=i+1}^{n-1}\left(\sum_{k>j\vee\ell }^n\mu_{i,j,k}\right)Y^{[j]}_{i+1}\cdots Y^{[j]}_{\ell-1} \left( Y^{[j]}_{\ell}-1\right)\notag\\
 &=X_i \cdot\sum_{\ell=i+1}^{n-1} \sum_{j=\ell+1}^{n-1}\left(\sum_{k=j+1 }^{n}\mu_{i,j,k}\right)Y^{[j]}_{i+1}\cdots Y^{[j]}_{\ell-1} \left( Y^{[j]}_{\ell}-1\right)+r_{i,1}\notag\\
 &=\mu_{i,\cdot,\cdot}^{(n)}X_i \cdot\left(\sum_{\ell=i+1}^{n-1} Y_{i+1}\cdots Y_{\ell-1} \left( Y_{\ell}-1\right)\right)+r_{i,1}-r_{i,2},\label{eq:condmean_dom}
\end{align}
where
\begin{align}
 r_{i,1} &:=\sum_{\ell=i+1}^{n-1} \sum_{j=i+1}^\ell\left(\sum_{k>\ell }^n\mu_{i,j,k}\right) X_iY^{[j]}_{i+1}\cdots Y^{[j]}_{\ell-1} \left( Y^{[j]}_{\ell}-1\right),\label{eq:r1}\\
 r_{i,2}&:= \sum_{\ell=i+1}^{n-1} \left( \mu_{i,\cdot,\cdot}^{(n)}- \sum_{k,j=\ell+1, k>j}^n\mu_{i,j,k}\right)\cdot X_iY_{i+1}\cdots Y_{\ell-1} \left( Y_{\ell}-1\right).\label{eq:r2}
\end{align}

Next, we approximate the sum in~\eqref{eq:condmean_dom} by an infinite product of the $Y_j$'s. Indeed,
\begin{align*}
~\eqref{eq:condmean_dom}= \mu_{i,\cdot,\cdot}^{(n)} X_i\cdot \left(\prod_{j=i+1}^{\infty}Y_j-1\right) +r_{i,1}-r_{i,2}-r_{i,3}-r_{i,4},
\end{align*}
where the error terms are
\begin{align}
 r_{i,3}&:= \mu_{i,\cdot,\cdot}^{(n)}X_i \cdot \left(\prod_{j=i+1}^{n-1} Y_j-1\right)\cdot \left(\prod_{j=n}^\infty Y_j-1\right),\label{eq:r3}\\
 r_{i,4}&:=\mu_{i,\cdot,\cdot}^{(n)}X_i\cdot \left(\prod_{j=n}^\infty Y_j-1\right).\label{eq:r4}
\end{align}
Therefore, we get
\begin{align}
\Delta_{n,i}-\E \Delta_{n,i}=\mu_{i,\cdot,\cdot}^{(n)} \left( X_i\prod_{j=i+1}^{\infty}Y_j-1\right) +r_{i,1}-r_{i,2}-r_{i,3}-r_{i,4}.
\end{align}
Recall $\tilde{s}_n$ and $b_{i}$ from~\eqref{eq:s_n} and~\eqref{def:bi}, respectively. From Lemma~\ref{lem:mu_ijk}, for fixed $i$ we have
$
\mu_{i,\cdot,\cdot}^{(n)} /\tilde{s}_n\to b_i
$
as $n\to\infty$.
Finally, let
\begin{align*}
Z:=\sum_{i=1}^\infty b_{i}\cdot \Bigl(X_i\prod_{j=i+1}^\infty Y_j -1 \Bigr).
\end{align*}
Then
\begin{align}\label{eq:final_error}
 \norm{\Delta_n-\E \Delta_n-\tilde{s}_{n}Z}_2
 &\le \tilde{s}_n\cdot \norm{\sum_{i=n-1}^\infty b_i\cdot \Bigl(X_i\prod_{j=i+1}^\infty Y_j -1 \Bigr)}_2 + \sum_{j=1}^4 \norm{\sum_{i=1}^n r_{i,j} }_2\notag\\
& \qquad + \norm{\sum_{i=1}^{n-2} \left(\mu_{i,\cdot,\cdot}^{(n)}-\tilde{s}_{n}b_i\right)\cdot \Bigl(X_i\prod_{j=i+1}^\infty Y_j -1 \Bigr)}_2.
\end{align}
Using Lemma~\ref{lem:mu_ijk}, we get that 
\begin{align}\label{eq:mu_approx}
    \abs{\mu_{i,\cdot,\cdot}^{(n)} - \tilde{s}_n b_i}\le C\cdot \frac{\tilde{s}_n}{\log n}\cdot b_{i}\log i
\end{align}
for all $i\ge 1$ and we get a bound of $O(\tilde{s}_n/\log n)$ for the last term in~\eqref{eq:final_error}.
The next four lemmas, Lemmas~\ref{lem:infiniteprod}--\ref{lem:ri4}, bound $r_{i,j}$ for $j=1,2,3,4$; Lemma~\ref{lem:infiniteprod} also justifies the infinite-product representation. These estimates complete the proof. The $\dwas$ bound follows from the definition of the Wasserstein-$1$ distance $\dwas$. Finally, we note that the value of $\gamma$ given in \eqref{eq:gamma} can be obtained by comparing means of $T_n$ and the limiting distribution using equations~\eqref{eq:mu_approx}.
\end{proof}

\begin{lem}\label{lem:infiniteprod}
If $\delta\le 0$, then
\begin{align*}
\norm{\sum_{i=n-1}^\infty b_i\cdot \Bigl(X_i\prod_{j=i+1}^\infty Y_j -1 \Bigr)}_2\lesssim n^{2\beta-3/2}\le n^{-1/2}.
\end{align*}
\end{lem}
\begin{proof}
Using independence and $\E X_i=\E Y_j=1$, Minkowski's inequality yields
\begin{align*}
&\norm{\sum_{i=n-1}^\infty b_i\bigl(X_iY_{i+1}Y_{i+2}\cdots -1\bigr)}_2\\
&\le \norm{\sum_{i=n-1}^\infty b_i(X_i-1)}_2
+ \sum_{i=n-1}^\infty b_i\norm{\sum_{j=i+1}^{\infty} X_iY_{i+1}\cdots Y_{j-1}(Y_j -1) }_2\\
&\lesssim \sqrt{\sum_{i=n-1}^\infty b_i^2\var(X_i)}
+\sum_{i=n-1}^\infty b_i\sqrt{\sum_{j=i+1}^{\infty} \E X_i^2 \E Y_{i+1}^2\cdots\E Y_{j-1}^2\var(Y_j)}.
\end{align*}
Since $b_i\cong i^{2\beta-2}$,
\begin{align*}
\sum_{i=n-1}^\infty b_i^2\var(X_i)\lesssim \sum_{i=n-1}^\infty i^{4\beta-4}\lesssim n^{4\beta-3}.
\end{align*}
Using the uniform second-moment product bound from Lemma~\ref{lem:bnds on Y} and $\var(Y_j)\lesssim j^{-2}$,
\begin{align*}
\sum_{j=i+1}^{\infty} \E X_i^2 \E Y_{i+1}^2\cdots\E Y_{j-1}^2\var(Y_j)\lesssim \sum_{j>i}\frac{1}{j^2}\lesssim \frac{1}{i}.
\end{align*}
Therefore,
\begin{align*}
\norm{\sum_{i=n-1}^\infty b_i\bigl(X_iY_{i+1}Y_{i+2}\cdots -1\bigr)}_2^2\lesssim n^{4\beta-3},
\end{align*}
which proves the claim.
\end{proof}

\begin{lem}\label{lem:ri1}
Fix $m\ge 2$ and $\delta\in(-m,0]$. Let $r_{i,1}$ be as defined in~\eqref{eq:r1}. Then
\begin{align*}
\norm{\sum_{i=1}^{n-2}r_{i,1}}_2\lesssim f_{n}.
\end{align*}
\end{lem}
\begin{proof}
Denote $
 \mu_{i,j,(\ell,n)} := \sum_{k>\ell }^n\mu_{i,j,k}.
$
We have
\begin{align*}
 \norm{\sum_{i=1}^{n-2}r_{i,1}}_2^2
 &\lesssim \sum_{\ell=2}^{n-1} \sum_{i<j\le \ell\atop i'<j'\le \ell}
 \mu_{i,j,(\ell,n)} \mu_{i',j',(\ell,n)} \cdot \cov(Y^{[j]}_{\ell},Y^{[j']}_{\ell}) \\
 &\lesssim
 f_n^2 \cdot\sum_{\ell=2}^{n-1} \sum_{i<j< \ell\atop i'<j'< \ell}
 \frac{\E \psi_i^2}{\Delta f_{i+1}}\cdot \frac{\E \psi_{i'}^2}{\Delta f_{i'+1}} \cdot \frac{1}{jj'}\cdot \ell^{-2} \\
 &\quad+2f_n^2 \cdot\sum_{\ell=2}^{n-1} \sum_{i<j< \ell\atop i'< \ell}
 \frac{\E \psi_i^2}{\Delta f_{i+1}}\cdot \frac{\E \psi_{i'}^2}{\Delta f_{i'+1}} \cdot \frac{1}{j\ell}\cdot \frac1{\ell} \\
 &\quad+ 2f_n^2 \cdot\sum_{\ell=2}^{n-1} \sum_{i< \ell\atop i'< \ell}
 \frac{\E \psi_i^2}{\Delta f_{i+1}}\cdot \frac{\E \psi_{i'}^2}{\Delta f_{i'+1}} \cdot \frac{1}{\ell^2}
 \lesssim
 f_n^2 \cdot\sum_{\ell=2}^{n-1} \sum_{i\le i'< \ell}
 (ii')^{2\beta-2}\cdot (\log \ell)^2\ell^{-2}.
\end{align*}
Finally, using $\beta\le 1/2$, we get that the last term is bounded by a constant multiple of $f_n^2$. This completes the proof.
\end{proof}

\begin{lem}\label{lem:ri2}
Fix $m\ge 2$ and $\delta\in(-m,0]$. Let $r_{i,2}$ be as defined in~\eqref{eq:r2}.
Then
\begin{align*}
\norm{\sum_{i=1}^{n-2}r_{i,2}}_2\lesssim f_n.
\end{align*}
\end{lem}
\begin{proof}
First, rewrite $r_{i,2}$ as
\begin{align*}
r_{i,2}= \sum_{\ell=i+1}^{n-1} \left( \sum_{i<j\le \ell<k\le n}\mu_{i,j,k}\right) X_iY_{i+1}\cdots Y_{\ell-1} \left( Y_{\ell}-1\right).
\end{align*}
Then, using independence of the $Y_t$'s and $\E Y_t=1$,
\begin{align*}
\E\abs{\sum_{i=1}^{n-2}r_{i,2}}^2
&\lesssim \sum_{i,i'\ge 1}\sum_{\ell\ge i\vee i'} \left( \sum_{i<j\le \ell<k\le n}\mu_{i,j,k}\right)\left( \sum_{i'<j'\le \ell<k'\le n}\mu_{i',j',k'}\right) \\
&\qquad \cdot \E(X_iX_{i'})\prod_{t=i\vee i' +1}^{\ell-1} \E(Y_t^2)\cdot \var(Y_\ell)\\
&\lesssim f_n^2\cdot \left(\sum_{1\le i'\le i \le n}\frac{(\log i)^2}{i^{3-2\beta}(i')^{2-2\beta}}\right)
\lesssim f_n^2.
\end{align*}
In the last step we used Corollary~\ref{lem:mu_ijk2}, the bound $\var(Y_\ell)\lesssim \ell^{-2}$, and $\sum_{\ell\ge t}(\log \ell)^2/\ell^2 \lesssim (\log t)^2/t$. This completes the proof.
\end{proof}

\begin{lem}\label{lem:varYprod}
We have
$
\var\left(\prod_{j=n}^\infty Y_j\right)\cong 1/n.
$
\end{lem}
\begin{proof}
We have
$
\prod_{j=n}^\infty Y_j=1+ \sum_{j=n}^\infty \prod_{k=n}^{j-1}Y_k\cdot (Y_j-1).
$
Using Lemma~\ref{lem:bnds on Y}, we get
\[
\var\left(\prod_{j=n}^\infty Y_j\right)=
\sum_{j=n}^\infty \prod_{k=n}^{j-1}\E Y_k^2\cdot \var(Y_j) \cong 1/n.
\]
This completes the proof.
\end{proof}

\begin{lem}\label{lem:ri3}
Fix $m\ge 2$ and $\delta\in(-m,0]$. Let $r_{i,3}$ be as defined in~\eqref{eq:r3}.
Then
\begin{align*}
\norm{\sum_{i=1}^{n-2}r_{i,3}}_2\cong s_n\cdot n^{-1/2}.
\end{align*}
\end{lem}
\begin{proof}
We have
\begin{align*}
 \sum_{i=1}^{n-2}r_{i,3}
 =
 \sum_{i=1}^{n-2}\mu_{i,\cdot,\cdot}^{(n)}\cdot \biggl(X_i-1 + X_i\sum_{k=i+1}^{n-1}Y_{i+1}\cdots Y_{k-1}(Y_k-1)\biggr)\cdot \left(\prod_{j=n}^\infty Y_j-1\right).
\end{align*}
Using Lemma~\ref{lem:bnds on Y},
\begin{align*}
 &\norm{\sum_{i=1}^{n-2}\mu_{i,\cdot,\cdot}^{(n)}\cdot \biggl(X_i-1 + X_i\sum_{k=i+1}^{n-1}Y_{i+1}\cdots Y_{k-1}(Y_k-1)\biggr)}_2^2 \\
 &= 
 \sum_{i=1}^{n-2} \underbrace{(\mu_{i,\cdot,\cdot}^{(n)})^2}_{(s_n/i^{2-2\beta})^2} \underbrace{\biggl(\var(X_i) + \E X_i^2\sum_{k=i+1}^{n-1}\E Y_{i+1}^2\cdots \E Y_{k-1}^2\var(Y_k)\biggr)}_{\text{uniformly bounded}}\\
 &\quad+ 2\sum_{i=1\atop i< i'}^{n-2}\sum_{k=i'+1}^{n-1} \underbrace{\mu_{i,\cdot,\cdot}^{(n)}}_{s_n/i^{2-2\beta}}\underbrace{\mu_{i',\cdot,\cdot}^{(n)}}_{s_n/(i')^{2-2\beta}}
 \underbrace{\E (X_iY_{i+1}\cdots Y_{i'-1})\E(X_{i'}Y_{i'}) \cdot \E (Y_{i'+1}^2\cdots Y_{k-1}^2)}_{\text{uniformly bounded}}\underbrace{\var(Y_k)}_{1/k^2}\\
 &\lesssim s_n^2,
\end{align*}
where, with under-braces, we denote the orders of the respective terms.
The claim follows by multiplying with $\norm{\prod_{j=n}^\infty Y_j-1}_2\cong n^{-1/2}$.
\end{proof}

\begin{lem}\label{lem:ri4}
Fix $m\ge 2$ and $\delta\in(-m,0]$. Let $r_{i,4}$ be as defined in~\eqref{eq:r4}.
Then
\begin{align*}
\norm{\sum_{i=1}^{n-2}r_{i,4}}_2\lesssim \frac{s_n}{\sqrt{n}}\cdot \log n.
\end{align*}
\end{lem}
\begin{proof}
Therefore,
\begin{align*}
 \norm{\sum_{i=1}^{n-2}r_{i,4}}_2
 &=\norm{\sum_{i=1}^{n-2}\mu_{i,\cdot,\cdot}^{(n)}X_i }_{2} \norm{\prod_{j=n}^\infty Y_j-1}_2\lesssim s_n\cdot \log n\cdot \frac{1}{\sqrt{n}}.
\end{align*}
This completes the proof.
\end{proof}

\section{Phase transition in $\delta $}\label{sec:phase_transition}
Let $m\ge 2$ be a fixed integer. Fix a sequence of real numbers $\mvgd=(\delta_t)_{t\ge 1}\in [-1,1]^\dN$. We modify the definition of the random network --  the network is constructed recursively as follows.
\begin{algorithmic}
 \State $\mathbf{n=1}$: $\mathrm{PAM}_1(m, \mvgd)$ consists of a single vertex and no edges.
 \State $\mathbf{n=2}$: $\mathrm{PAM}_2(m, \mvgd)$ consists of two vertices with $m$ edges between them.
 \State $\mathbf{n=k}$:~Given $\mathrm{PAM}_{k-1}(m, \mvgd)$, the $k$-th vertex is added with $m$ edges $\{e_{k,\ell}\}_{\ell\in[m]}$. The edges connect sequentially to vertices of $\mathrm{PAM}_{k-1}(m, \mvgd)$ with probabilities proportional to an affine function of the current degree, which is updated after each new connection; in other words, for $\ell\in[m-1]$ and $v\in[k-1]$,
 \[
 \pr(e_{k,\ell+1} \text{ connects to } v)=
 \frac{\deg^{(\ell)}_{k-1}(v)+\delta_v}{2m(k-2) + \ell + \sum_{i=1}^{k-1}\delta_i},
 \]
 where $\deg_{k-1}^{(\ell)}(v)$ denotes the degree of vertex $v\in \mathrm{PAM}_{k-1}(m, \mvgd)$ after $\ell$ edges from the $k$-th vertex have been connected to the graph.
\end{algorithmic}

The above graph process can also be described by the P\'olya urn graph introduced in~\cite{urngraph} with a little modification. 
Let $\psi_1\equiv 1$ and, for all $i\ge 2$, let $\psi_i$ be independent $\mathrm{Beta}(m+\delta_i,m(2i-3)+\sum_{j=1}^{i-1}\delta_j)$ random variables. Define
\begin{align*}
\varphi_{j,n}:=\psi_j\prod_{t=j+1}^n(1-\psi_t),\quad S_{(k,n]}:=\sum_{j=1}^k\varphi_{j,n}=\prod_{t=k+1}^n(1-\psi_t), \quad I_k^{(n)}:=\left(S_{(k-1,n]},S_{(k,n]}\right].
\end{align*}
Conditional on $\mvpsi:=(\psi_i)_{i\ge 1}$, let 
\begin{align*}
\{ U_{k,\ell}\}_{k\in[n],\ell\in[m]} \text{ be independent and uniformly distributed on } [0,S_{(k-1,n]}].
\end{align*}
For any $1\le j<k\le n$, place an edge from $k$ to $j$ for every $U_{k,\ell}$ contained in $I_j^{(n)}$ for some $\ell\in[m]$. The random graph constructed in such a manner has the same distribution as the sequential linear preferential attachment model $\pam(m,\mvgd)$.

\begin{proof}[Proof of Theorem~\ref{thm:phase_transition}]
In particular, we have the conditional expectation given $\mvpsi$ of the number of triangles with $i$ as the lowest labeled vertex is $\E(T_{n,i}\mid \mvpsi)=m^2(m-1) \Delta_{n,i}$ for $i\ge 1$ where
\begin{align}
\Delta_{n,i}:=\psi_{i}^{2} \sum_{i<j<k\le n} \prod_{t=i+1, t\neq j}^{k-1}(1-\psi_{t})^{2}\cdot \psi_{j} (1-\psi_{j})
\end{align}
with mean
\begin{align}
S_{n}(i):=\E\Delta_{n,i}=\frac{\E \psi_i^2}{\prod_{t=2}^{i} \E(1-\psi_t)^2} \sum_{i<j< k\le n} \prod_{t=2}^{k-1} \E(1-\psi_t)^2\cdot \frac{\E\psi_j(1-\psi_j)}{\E(1-\psi_j)^2}.
\end{align}
Define, $f_1=g_1=0,h_1=1$ and for $n\ge 2$
\[
f_{n}:=1+\sum_{k=2}^{n-1}\prod_{t=2}^{k} \E(1-\psi_t)^2, \quad
g_{n}:= \sum_{j=2}^{n} \frac{\E\psi_j(1-\psi_j)}{\E(1-\psi_j)^2},\quad
h_{n}:=\sum_{i=1}^{n} \frac{\E \psi_i^2}{\Delta f_{i+1}},
\]
so that we have with $\Delta f_{n}:=f_{n}-f_{n-1},$ $\Delta  g_{n}:=g_{n}-g_{n-1},$ $\Delta  h_{n}=h_{n}-h_{n-1}$,
\[
\E\Delta_{n,i} = \sum_{i<j< k\le n} \Delta  h_i\cdot \Delta  g_j\cdot \Delta f_k.
\]
We fix $c \in \dR$, and $\alpha>0$. Define $\delta _1\equiv 0$ and 
$
\delta_t := 2mc \cdot (\log t)^{-\alpha}\text{ for }t \ge 2.
$
Thus, we have 
$$
\log \E(1-\psi_{t})^{2} = -\frac{2}{t}\cdot \frac{m+\delta_t}{2m+ \frac1t\sum_{i=1}^t \delta _i } +O(1/t^{2}) = -\frac1t - \frac{2\delta_t - \frac1t\sum_{i=1}^t \delta _i}{2mt+\sum_{i=1}^t \delta _i } + O(1/t^{2}), \quad t\ge 2.
$$
Note that, $\frac1n\sum_{i=1}^n \delta _i = \Delta_n\cdot (1+O(1/\log n))$, $\E \psi_n^2\cong n^{-2}$ and $\Delta  g_n\cong n^{-1}$.
In particular,
\begin{align*}
\log \Delta f_n &= -\sum_{t=2}^{n-1}\frac1t - \sum_{t=2}^{n-1}\frac{2\delta_t - \frac1t\sum_{i=1}^t \delta _i}{2mt+\sum_{i=1}^t \delta _i } +O(1)\\
&=  -\log n - c\int_{2}^{n}\frac{dx}{x(\log x)^{\alpha}} +O\left(1+\int_{2}^{n}\frac{dx}{x(\log x)^{1+\alpha}} \right)\\
&=  -\log n - c\int_{2}^{n}\frac{dx}{x(\log x)^{\alpha}} +O(1).
\end{align*}

\noindent\textbf{Case $\alpha>1$.}
First, we study the case of $\alpha>1, c\in\dR$. We have $\Delta f_{n}\cong n^{-1}$ and thus 
\begin{align*}
\sum_{i=1}^{n}\E \Delta_{n,i} \cong \sum_{1\le i<j<k\le n} \frac{1}{i}\cdot \frac{1}{j}\cdot \frac{1}{k} \cong (\log n)^{3}.
\end{align*}
Moreover, we get $S_n(i) \cong {(\log n)^2}/{i}$ for all $i$. This behavior is the same as the $\delta =0$ case.\\

\noindent\textbf{Case $\alpha=1$.}
Next, we consider the $\alpha=1, c\neq 0$ case. We have $\Delta f_{n}\cong n^{-1}(\log n)^{-c}$. Thus, 
\begin{align*}
\sum_{i=1}^{n}\E \Delta_{n,i} \cong \sum_{1\le i<j<k\le n} \frac{(\log i)^c}{i}\cdot \frac{1}{j}\cdot \frac{1}{k(\log k)^c}.
\end{align*}
Depending on whether $c>-1$, $c=-1$, or $c<-1$, we get different behavior for the sum. For $c>-1$, we get $(\log n)^{3}$; for $c=-1$, we get $(\log n)^{3}\log \log n$; and for $c<-1$, we get $(\log n)^{2-c}$.
Moreover, we have
\[
S_n(i) \cong
\begin{cases}
\frac{(\log i)^2}{i}, & c>2,\\
\frac{(\log i)^2}{i}\cdot (\log\log n-\log\log i_0), & c=2, \\
\frac{(\log i)^c}{i}\cdot (\log n)^{2-c}, & c<2.
\end{cases}
\]

\noindent\textbf{Case $\alpha<1$.}
Finally, we fix $\alpha\in(0,1)$.  We have 
\[
\Delta f_{n}\cong \frac1n \exp\left(-\frac{c}{1-\alpha}\cdot (\log n)^{1-\alpha}\right).
\]
Thus,
\begin{align*}
S_{n}(i) &=\sum_{i<j< k\le n} \Delta  h_i\cdot \Delta  g_j\cdot \Delta f_k\\
&\cong \sum_{i<j< k\le n} \frac1{ijk}\cdot \exp\left(-\frac{c}{1-\alpha}\cdot (\log k)^{1-\alpha} +\frac{c}{1-\alpha}\cdot (\log i)^{1-\alpha} \right)\\
&\cong \sum_{i< k\le n} \frac{\log(k/i)}{ik}\cdot \exp\left(-\frac{c}{1-\alpha}\cdot (\log k)^{1-\alpha} +\frac{c}{1-\alpha}\cdot (\log i)^{1-\alpha} \right)
\end{align*}
First, we focus on the case $c>0$. We have
\begin{align*}
i\cdot S_{n}(i) 
& \cong  \int_{i}^{n}  \frac{\log(x/i)}{x} \exp\left( - \frac{c}{1-\alpha} ((\log x)^{1-\alpha} - (\log i)^{1-\alpha})\right)\!dx\\
&= \int_{0}^{\log(n/i)}  u\cdot \exp\left( - \frac{c}{1-\alpha} ((u+\log i)^{1-\alpha} - (\log i)^{1-\alpha})\right)\!du
\cong (\log i)^{2\alpha}.
\end{align*}
Here, we used the fact that for $b,u>0,\alpha\in (0,1)$, we have
\[
\frac{u}{b^\alpha+u^\alpha} \le  \int_0^u (b+x)^{-\alpha} dx =\frac{(b+u)^{1-\alpha}-b^{1-\alpha}}{1-\alpha}
\le \frac{u}{b^\alpha}.
\]
Thus, with $b=\log i\ge 1$, the integral satisfies
\[
b^{2\alpha}\int_{0}^{(\log n -b)/b^\alpha}  u e^{-cu}du
\le
\int_{0}^{\log n -b}  u \exp\left( - c\cdot \frac{(b+u)^{1-\alpha} - b^{1-\alpha}}{1-\alpha} \right)\!du
\le b^{2\alpha}\int_{0}^{\infty}  u \exp\left( -  \frac{cu}{1+u^\alpha} \right)\!du.
\]
In particular, we get for $c>0$,
\[
S_n(i)
\cong \frac{1}{i}
(\log i)^{2\alpha}
\]
and 
\[
\sum_{i=1}^{n}\E \Delta_{n,i} \cong (\log n)^{1+2\alpha}.
\]
For $c<0$, we get 
\[
f_n\cong F(\log n)\cong (\log n)^\alpha \exp(-c/(1-\alpha)\cdot (\log n)^{1-\alpha})  
\]
where
\begin{align*}
    F(x) &=\int_0^x \exp(-c/(1-\alpha) y^{1-\alpha})dy
\\
&=\exp(-c/(1-\alpha) x^{1-\alpha})\int_0^x \exp(c/(1-\alpha) (x^{1-\alpha}-(x-u)^{1-\alpha}))du\\
&\cong \exp(-c/(1-\alpha) x^{1-\alpha})\ (\log n)^\alpha \text{ when } x\gg1
\end{align*}
and
\begin{align*}
    \E\Delta_{n,i}
    = \Delta  h_i\cdot\sum_{i<j< k\le n}  \Delta  g_j\cdot \Delta f_k
    &= \Delta  h_i\cdot\sum_{i<j< n} \Delta  g_j\cdot (f_n - f_j)\\
    &= \Delta  h_i\cdot \biggl(f_n\sum_{i<j< n} \Delta  g_j -\sum_{i<j< n} f_j\Delta  g_j\biggr)
    \cong \Delta  h_i \cdot f_n\cdot \log n.
\end{align*}
Note that, $\Delta  h_i\cong \frac{1}{i}\exp(c/(1-\alpha)\cdot (\log i)^{1-\alpha})$ is summable when $c<0$. Thus we get
\[
\sum_{i=1}^{n}\E\Delta_{n,i} \cong f_n\cdot \log n \cong (\log n)^{1+\alpha} \exp\left(-c/(1-\alpha)\cdot (\log n)^{1-\alpha}\right).
\]
This completes the proof.
\end{proof}
As a corollary of the proof, we get the following.
\begin{cor}\label{prop:phase}
As $n \to \infty$, the asymptotics of $S_n(i)$ are as follows.
\begin{itemize}
\item \textbf{Case $\alpha>1$.} $S_n(i) \cong \frac{(\log n)^2}{i}$ for all $c\in\dR$.
\item \textbf{Case $\alpha=1$.} 
\[
S_n(i) \cong
\begin{cases}
\frac{(\log i)^2}{i}, & c>2,\\
\frac{(\log i)^2}{i}\cdot (\log\log n-\log\log i), & c=2, \\
\frac{(\log i)^c}{i}\cdot (\log n)^{2-c}, & -\infty<c<2.
\end{cases}
\]

\item \textbf{Case $0<\alpha<1$.}
\[
S_n(i) \cong 
\begin{cases}
\frac{(\log i)^{2\alpha}}{i}, &c>0,\\
\frac{1}{i}
\exp\!\big(\frac{|c|}{1-\alpha}\cdot [(\log n)^{1-\alpha}-(\log i)^{1-\alpha}]\big), & c<0.
\end{cases}
\]
\end{itemize}
\end{cor}

\begin{rem}\label{rem:scaling}
    In light of an Open Question~\ref{ques:phase_trans}, the computations above actually suggest the correct scaling in each regime. After considering the same centered edge decomposition as in~\eqref{eq:centered_decomp} we expect the correct scaling for $\Delta_n=\sum_{i=1}^{n-2}\Delta_{n,i}$ (given by~\ref{term:central0}) to be $(\log n)^{2-c}$ when $\alpha=1,c<2$. However, the scaling for $T_{n}^{(3)}$ is $\sqrt{\E T_n}\cong \sqrt{\E \Delta_n} \cong (\log n)^{3/2}$ for $\alpha=1, c>-1$. Thus the term dominating the fluctuation contribution changes from $\Delta_n$ to $T_{n}^{(3)}$ at $\alpha=1,c=1/2$ when $2-c=3/2$.
\end{rem}

\section{Closing remarks and simulations}\label{sec:closing}
\subsection{Interpretation of the limiting distribution}\label{sec:inerpr_lim}
For $\delta <0$, the limiting random variable from Theorem~\ref{thm:main}
\[
\sum_{i=1}^\infty b_{i}\cdot \frac{\psi_i^2}{\norm{\psi_i}_2^2}\prod_{j=i+1}^\infty \frac{(1-\psi_j)^2}{\norm{1-\psi_j}_2^2},
\]
can be written as $RS$, where 
\[
R=\prod_{i=2}^\infty (1-\psi_j)^{2}\norm{1-\psi_j}_2^{-2}\qquad \text{ and }\qquad S=1+\sum_{i=2}^\infty \psi_i^2 \prod_{j=2}^{i} (1-\psi_j)^{-2}.
\]

Our interpretation of the quantities $R$ and $S$ is conjectural, supported by empirical evidence. One way of making it rigorous requires the adaptation of the graphical computation lemma to the form that allows to restrict the count subgraphs to those that contain particular vertex $i$ as the oldest vertex (in the sense that in~ \eqref{eq:graphic_fla} $i_1$ is treated as a fixed parameter and the sum is only over $i_j$ for $j\ge2$) and following the rest of our argument.

Let $T_{n,i}$ be the number of triangles with the oldest vertex being $i$. Then, we conjecture that
\begin{equation}\label{eq:T_ratio}
    \left(\frac{T_{n,i}}{T_{n,1}}\right)_{i\ge 2}\xrightarrow{\text{\,\,\,d\,\,\,}}\left(\psi_i^2\prod_{j=2}^{i}(1-\psi_j)^{-2}\right)_{i\ge 2} \text{ as } n\to\infty,
\end{equation}
and that 
\begin{equation}\label{eq:T_ratio2}
    \frac{T_{n,1}}{\tau_n}\xrightarrow{\text{\,\,\,d\,\,\,}}R \quad n\to\infty,
\end{equation}
where $\tau_n:=m^2(m-1)\cdot \beta/(1-2\beta)\cdot n^{1-2\beta}\log n$.
\begin{figure}[htb]
\centering
\includegraphics[width=\linewidth]{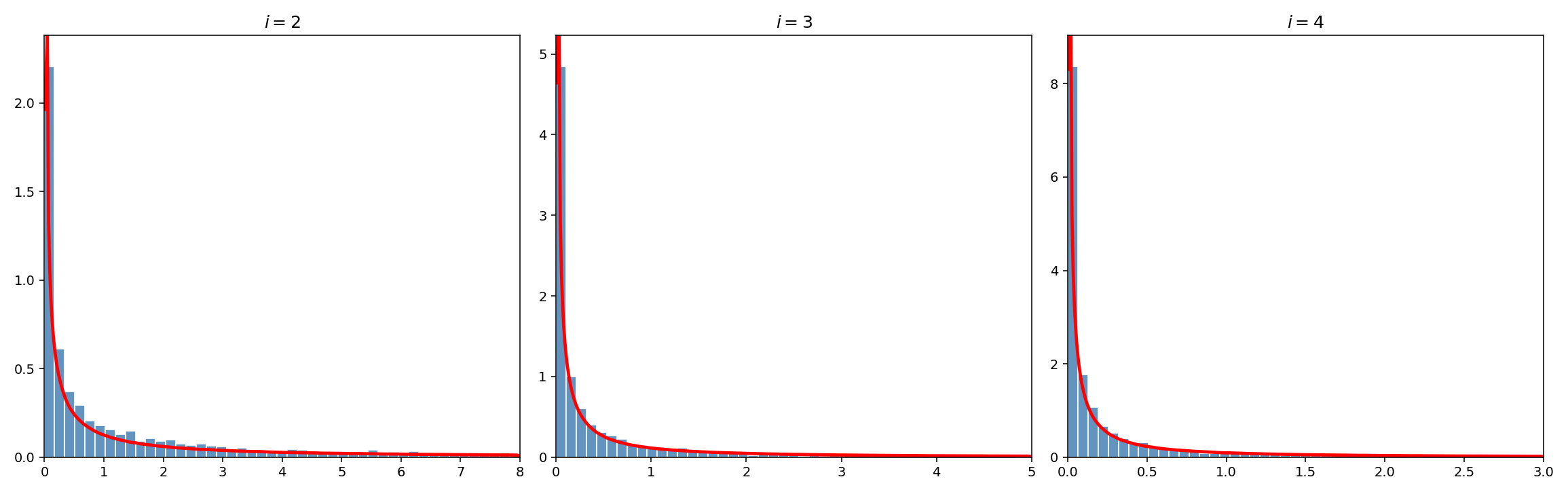}
\caption{This simulation consists of $5000$ samples of $\pam(m,\delta )$ with $n=10^6$, $m=2$, and $\delta=-1$. The figure depicts overleaped histograms of $T_{n,i}/T_{n,1}$ for $i=2,3,4$ [in blue] and the conjectured limiting distribution given in the right-hand side of~ \eqref{eq:T_ratio} [in red], \label{fig:T_ratios}}
\end{figure}
\begin{oques}\label{ques:ratio}
    Verify the distributional convergence from \eqref{eq:T_ratio} and \eqref{eq:T_ratio2} and show that all these random variables converge jointly to recover the limiting random variable from Theorem~\ref{thm:main}.
\end{oques}

\subsection{Stability of results in other variants}

As mentioned in the introduction, different variants of linear preferential attachment often yield comparable mean behavior across many statistics of interest. We already saw an example of this in Theorem~\ref{thm:ET} (proved in Section~\ref{sec:graphical_comp}). Indeed, the original results of Bollob\'as and Riordan~\cite{BollobasRiordan} for $\delta=0$ were derived for the variant with instantaneous attachment, that is, where all edges of the $n$-th vertex attach to the existing graph with probabilities given by
\begin{equation}\label{iLPAM}
 \pr(e_{k,\ell+1} \text{ connects to } v)=
 \frac{\deg_{k-1}(v)+\delta}{2m(k-2) + (k - 1)\delta};
\end{equation}
whereas for $\delta>0$, Eggemann and Noble~\cite{EggemannNoble} worked with a different sequential attachment model, in which the graph is constructed from a preferential-attachment tree, that is, each incoming vertex has a unique outgoing edge, and then blocks of $m$ vertices are collapsed.

In contrast, it is not clear whether this remains true for distributional convergence, or even for the order of fluctuations. We present empirical data for the (scaled and centered) triangle counts in the instantaneous-attachment variant. For $\delta>0$, Figure~\ref{fig:clt_del_pos} is consistent with the Gaussian limit; however, it is likely that the rate of convergence is of order $1/ \log n$ (similar to the case~\ref{thm:main_delta_pos} of Theorem~\ref{thm:main}) and thus these simulations do not rule out a non-Gaussian limit.

\begin{figure}[htb]
\centering
\begin{subfigure}[b]{0.3\textwidth}
  \centering
  \includegraphics[width=\linewidth]{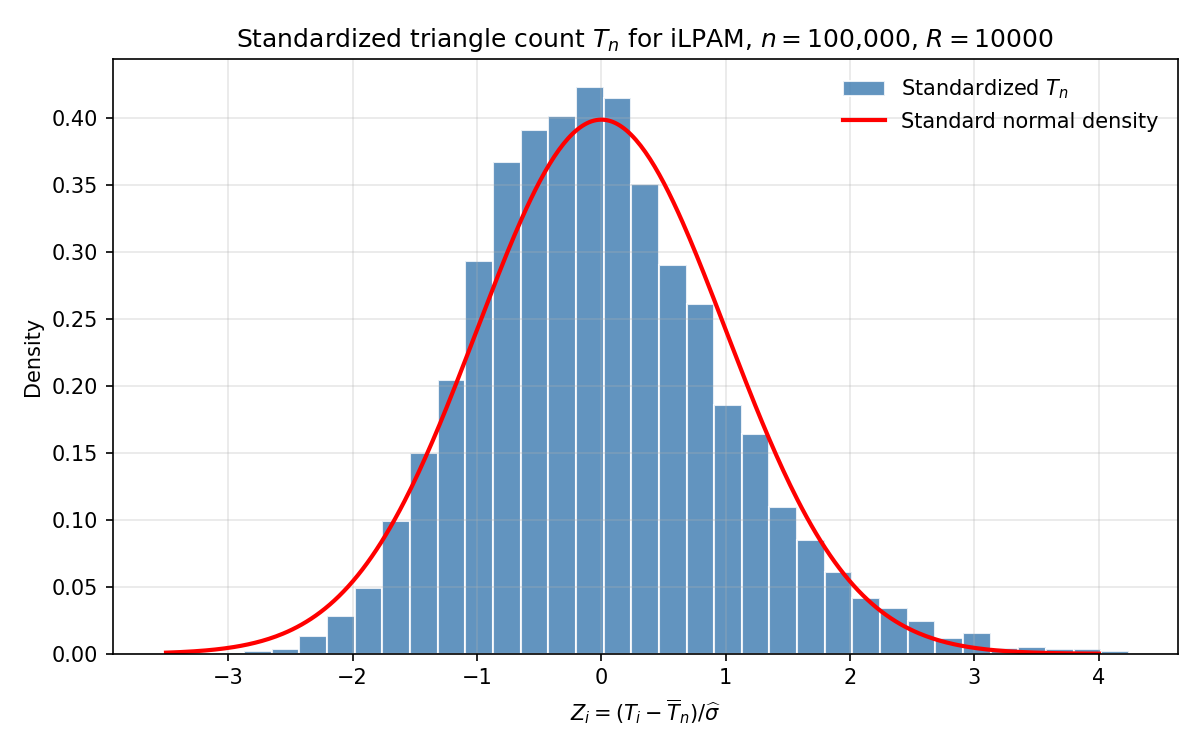}
\end{subfigure}
\hfill
\begin{subfigure}[b]{0.3\textwidth}
  \centering
  \includegraphics[width=\linewidth]{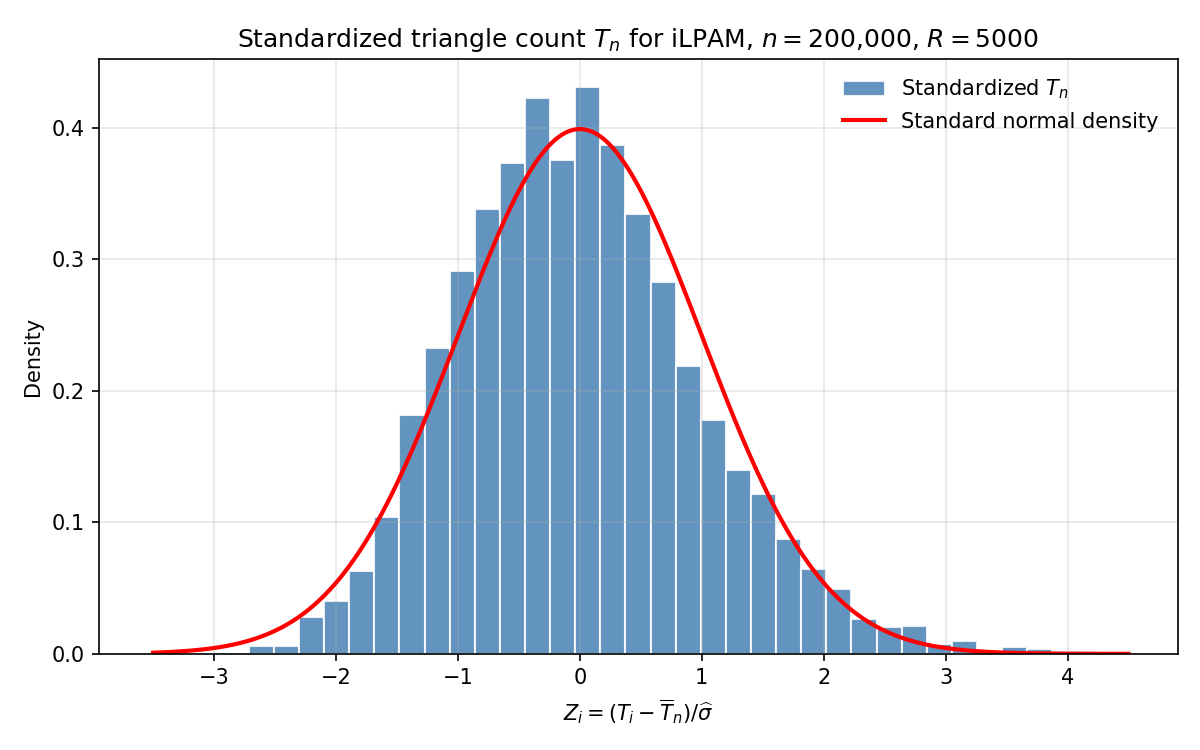}
\end{subfigure}
\hfill
\begin{subfigure}[b]{0.3\textwidth}
  \centering
  \includegraphics[width=\linewidth]{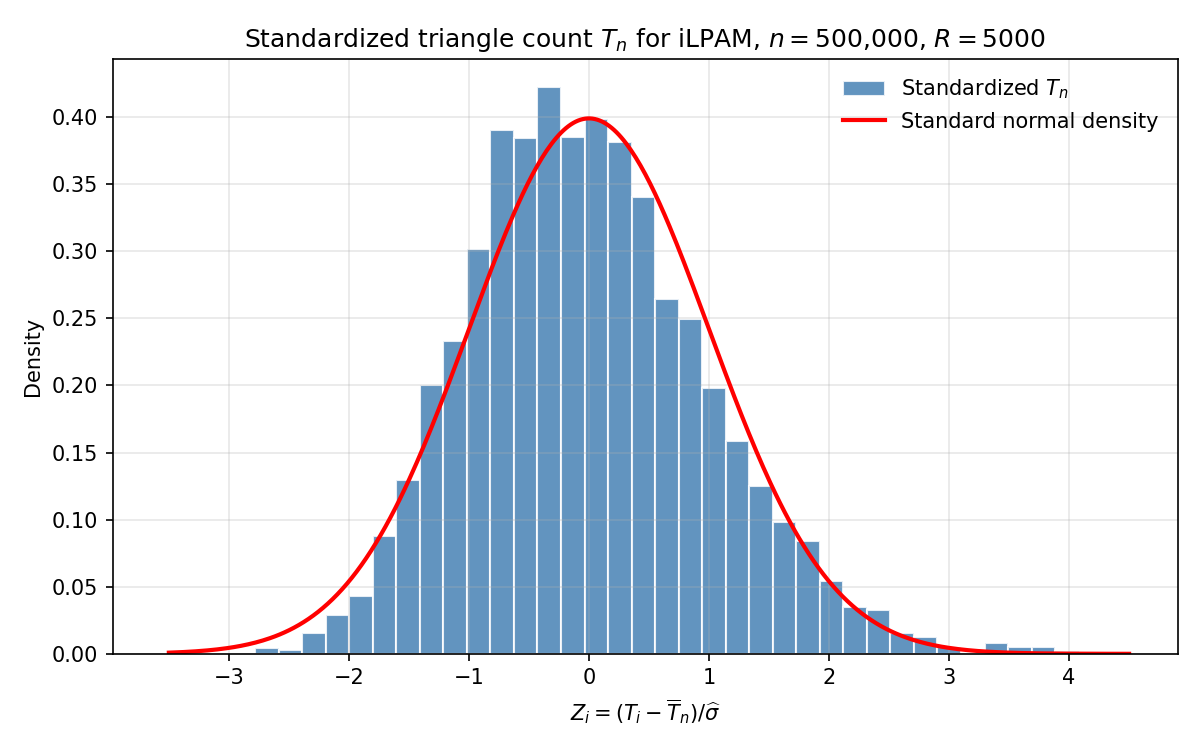}
\end{subfigure}

\par\medskip
\begin{subfigure}[b]{0.30\textwidth}
  \centering
  \includegraphics[width=\linewidth]{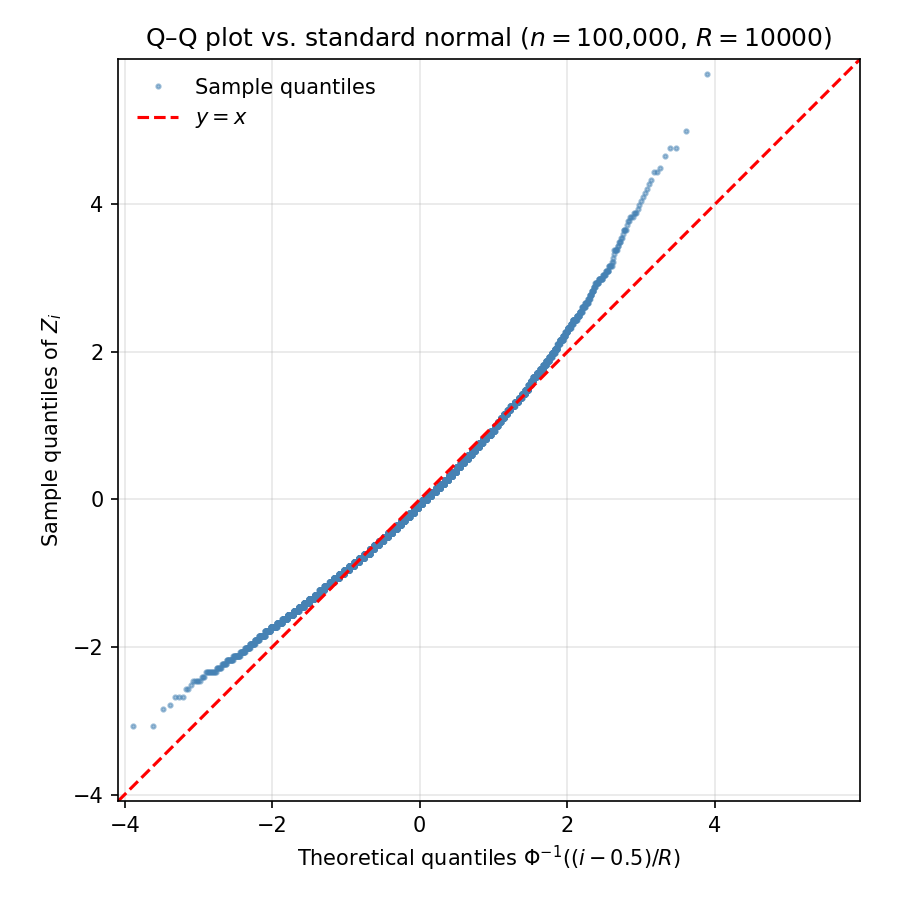}
\end{subfigure}
\hfill
\begin{subfigure}[b]{0.30\textwidth}
  \centering
  \includegraphics[width=\linewidth]{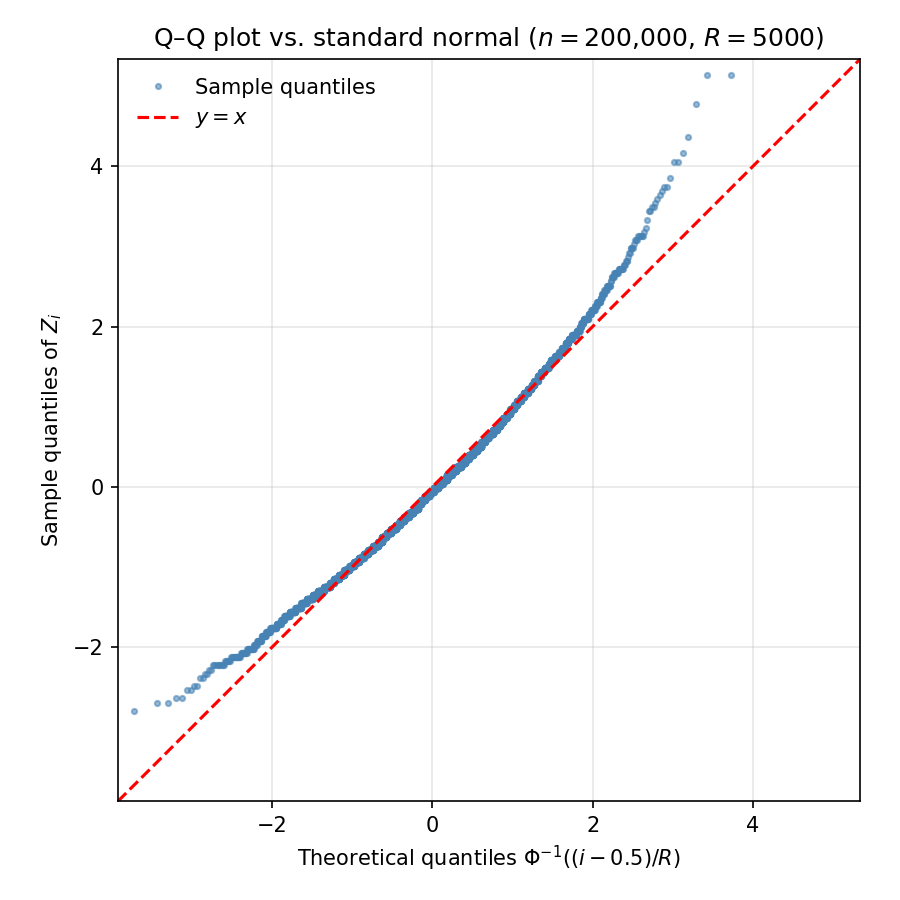}
\end{subfigure}
\hfill
\begin{subfigure}[b]{0.30\textwidth}
  \centering
  \includegraphics[width=\linewidth]{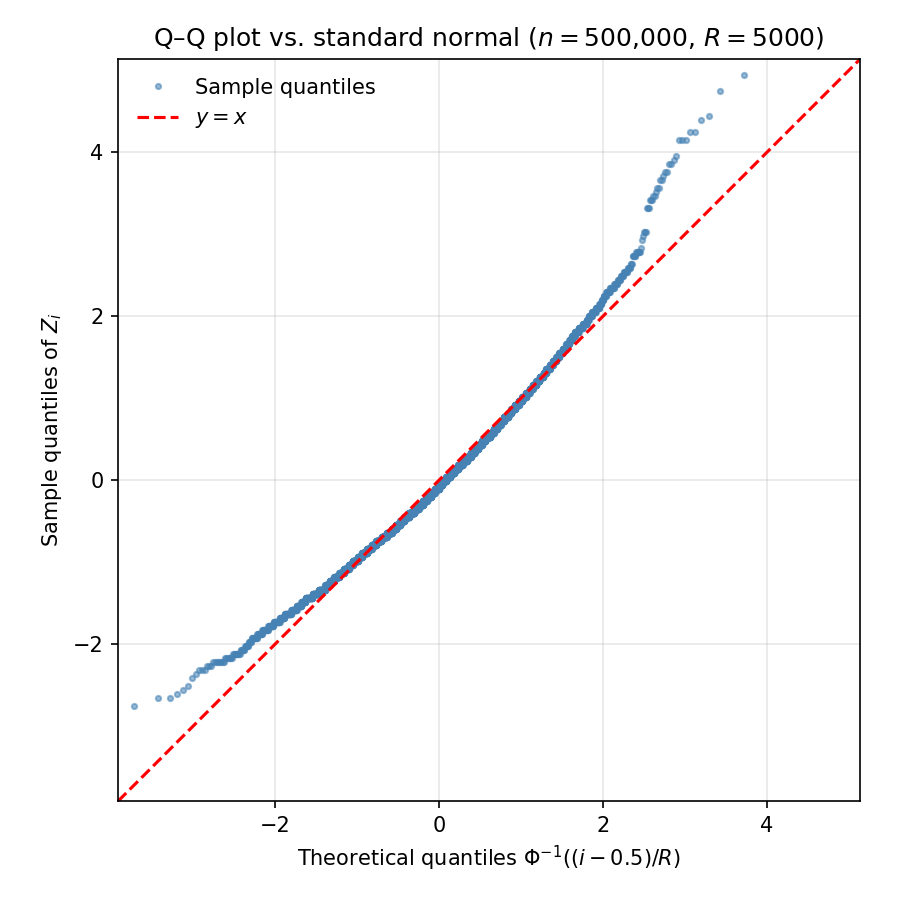}
\end{subfigure}

\caption{Histograms and QQ plots for centered and scaled $T_n$ in the instantaneous variant of the linear preferential attachment model with $m=2$ and $\delta=1$. Here $n=10^5$ [left], $n=2\cdot 10^5$ [middle], and $n=5\cdot 10^5$ [right] denote the number of vertices, while $R=10000$ [left] and $R=5000$ [middle and right] denote the corresponding number of samples.}
\label{fig:clt_del_pos}
\end{figure}
\newpage
\begin{conj}
     Let $m\ge2$ and $\delta>0$ be fixed. Let $T_n$ be the number of triangles in the instantaneous variant of the linear preferential attachment model,  that is, the connections are as in~\eqref{iLPAM}. Then 
 \begin{align*}
 \dwas\left(\frac{T_n-\E T_n}{\sqrt{\var(T_n)}}, Z\right)\lesssim \frac{1}{\sqrt{\log n}},
 \end{align*}
 where $Z\sim \N(0,1)$.
\end{conj}

In general, it is unlikely that the limiting distributions from parts~\ref{thm:main_delta_zero} and~\ref{thm:main_delta_neg} of Theorem~\ref{thm:main} are universal. We believe that it is likely that for other variants the broader structural behavior is similar, \ie~that the dominating contribution comes from triangles where one vertex is old, and the rest were recently added, while the exact distribution differs both for the total number of triangles and for quantities such as \eqref{eq:T_ratio}. Surprisingly, empirical data do not suggest a significant difference between the instantaneous and sequential attachment variants of LPAM.

\begin{figure}[htb]
\centering
\includegraphics[width=\linewidth]{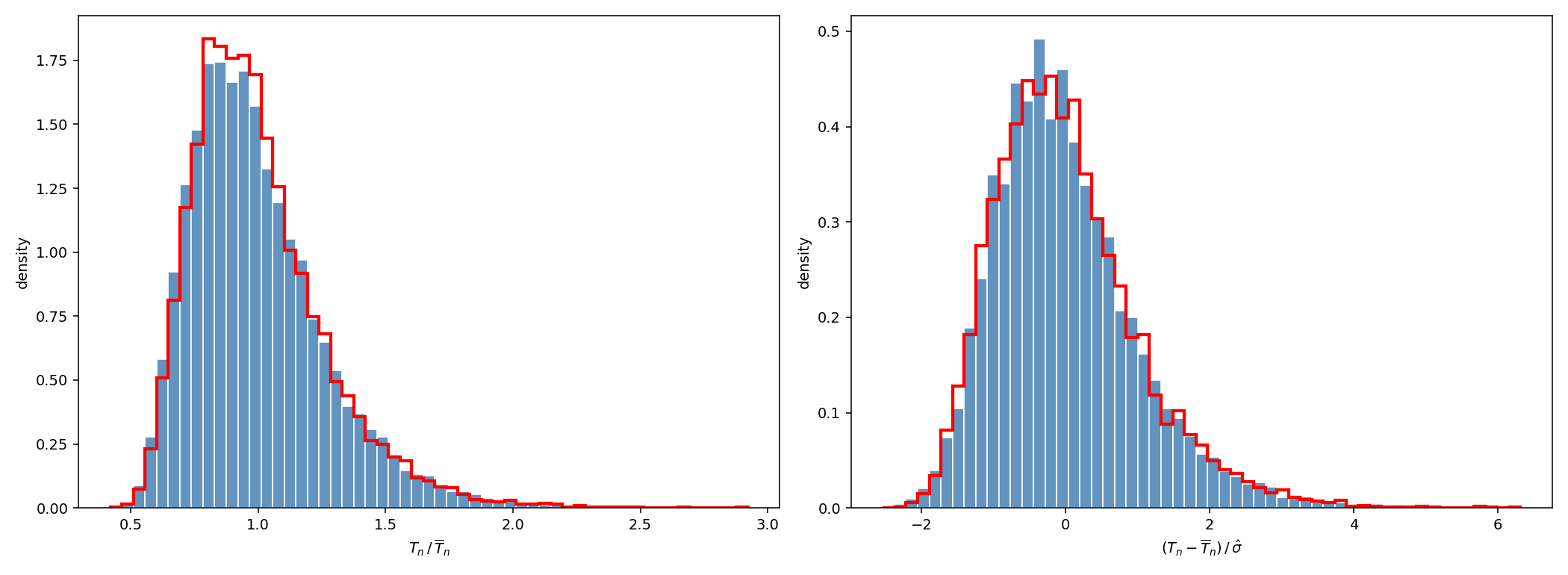}
\caption{Comparison between histogram for triangle counts in the instantaneous [red line] and sequential [in blue bars] attachment variants of LPAMs after appropriate scaling and centering following Theorem~\ref{thm:main}. The left figure is for $\delta=-1$ while the right one is for $\delta=0$. All simulations were performed for $n=10^5$, $m=2$, and $10^4$ samples for each model. \label{fig:clt_del_nonpos}}
\end{figure}
\begin{oques}\label{ques:clt_del_nonpos}
    Find the exact distributional limit for the instantaneous linear preferential attachment model for $\gd\le0$. 
\end{oques}

\subsection{Other open questions}

\begin{oques}\label{ques:cycle}
    Derive distributional limits for the number of cycles of fixed length $2k$ or $2k+1$, which we conjecture to be Gaussian when $\delta >0$ and for $\delta <0$ to be given by
    \[ \sum_{1\le i_1<i_2<\cdots<i_k}b_{i_k}\prod_{j=1}^k\left(\psi_{i_j}^2\prod_{t=i_j+1}^{i_{k}}(1-\psi_t)^2\right)\prod_{t=i_{k}+1}^\infty\frac{(1-\psi_t)^{2k}}{\E(1-\psi_t)^{2k}},\]
    where 
    \[
\frac{1}{b_{i_k}}=\E\prod_{j=1}^k\left(\prod_{t=j+1}^{i_{k}}(1-\psi_t)^2\right).
    \]
\end{oques}

\begin{oques}\label{ques:cliques}
    Derive distributional limits for the number of cliques of fixed size $(k+1)$ for $\delta <0$. We conjecture it to be 
     \[ \sum_{i=1}^\infty \frac{\psi_{i}^k}{\prod_{t=2}^{i}\E(1-\psi_t)^k}\cdot \prod_{t=i+1}^{\infty}\frac{(1-\psi_t)^{k}}{\E(1-\psi_t)^{k}}.\]
\end{oques}

\begin{oques}\label{ques:phase_trans}
    Analyze the continuity of the phase transition, analogous to Theorem~\ref{thm:phase_transition}, for 
    \begin{enumerate}
        \item the variance of $T_n$, which we conjecture fixing $\alpha=1$ it occurs at $c=1/2$, see Remark~\ref{rem:scaling};
        \item the distributional limit for $T_n$;
        \item means of other subgraphs, such as cycles or cliques of fixed size.
    \end{enumerate}
\end{oques}

\begin{oques}\label{ques:ldp}
    Derive a large deviation principle for triangles in the linear preferential attachment models and corresponding structural phase transitions.
\end{oques}

\section*{Acknowledgments}

We would like to thank Remco van der Hofstad for introducing the second author to the topic and pointing towards the question of establishing the limiting distribution of triangle counts in linear PAM. We thank Shankar Bhamidi for many insightful conversations and his support. We also thank Anant Raj for helpful discussions at the beginning of the project.
The first author was supported in part by CRB Grant RB23016.
The second author was supported in part by the RTG award grant (DMS-2134107) from the NSF and by the ERC Synergy Grant No. 810115 - DYNASNET. Simulations in Figures~\ref{fig:T_ratios}--\ref{fig:clt_del_nonpos} were done using Claude's Opus 4.7 engine. 

\appendix
\section{Edge and triangle counts in RGIV}\label{sec:appendix_RGIV}
In this Section, we briefly consider a different model, the random graph model with immigrating vertices, introduced by Aldous and Pittel in~\cite{RGIV}, to highlight the use of the centered-edge decomposition and the competition between terms in the context of distributional convergence for small subgraph counts. In this model, the graph is also constructed in continuous time as follows.
\begin{algorithmic}
 \State \textbf{Initial data:} $\lambda:=\lambda(t):(0,\infty)\to(0,\infty)$, $G_0=\emptyset$.
 
 \State \textbf{Vertices:} The vertices arrive according to a Poisson process with rate $\lambda$.
 
 \State \textbf{Edges:} The edge arrival process between two vertices $i<j$ already present in the graph is given by i.i.d.~Poisson processes of rate $1/\lambda$.
\end{algorithmic}
  We observe the graph at time $t$ where $\lambda\to\infty, \lambda t \to \infty$.

\subsection{Edge count}
Let $N_t$ be the number of vertices at time $t$ with $N_t\sim \poi(\lambda t).$ Let $\cE_t$ be the number of edges at time $t$, respectively. Define
 \[
        \theta:={t}/{\lambda},\qquad \mu_t:=\E \cE_t\text{ and } \gs_t^2=\var(\cE_t).
\] 
Throughout the section, we assume that $\lambda\to\infty$. We will prove the following.

\begin{thm}[Limit Theorems for the edge count]\label{thm:RGIV_E}
    The following convergences hold as $\lambda\to\infty$. 
    \begin{enumerate}[label=\upshape (\roman*)]
        \item\label{thm:RGIV_E_Poi}
        When $\lambda t^{3}\to c\in (0,\infty)$, we have 
        \[
        \dtv(\cE_t,\poi(c/6))\lesssim t/\lambda + t(t+\sqrt{\lambda t^3}) + \abs{\lambda t^{3}-c}.
        \]
       
        \item\label{thm:RGIV_E_N1} When $\lambda t^3\to\infty$ and $t=O(\lambda)$, we have 
        \[
        \dwas\left((\cE_t-\mu_t)/\gs_t,\N(0,1)\right)\lesssim \frac{1}{\gs_{t}} 
        \]
        where
        \[
        \gs_{t}^2=(1+o(1))\cdot \begin{cases}
            \frac{1}{6}\lambda t^3 + \frac{2}{15}\lambda t^5 &\text{ when } t\ll \lambda\\
            C_\theta\cdot\lambda^3 t^3 &\text{ when }t\approx \lambda,
        \end{cases}
        \]
        for a constant $C_{\theta}$ depending on $\theta=t/\lambda$.
    \end{enumerate}
\end{thm}

We provide a skeleton of the proof of Theorem~\ref{thm:RGIV_E}, which crucially depends on the decomposition of $\cE_{t}$ into a sum of uncorrelated random variables. Depending on the assumption on $t=t_{\lambda}$, only a few terms will dominate.
First, we note that conditionally on $N_t=n$, the arrival times are distributed as the order statistics of $n$ many i.i.d.~$\mathrm{Uniform}(0,t)$ random variables. Let $\{\xi_{ij}\}_{i,j\ge1}$ be i.i.d.~$\textrm{Exp}(1)$ and $\{U_i\}_{i\ge1}$ be i.i.d.~$\mathrm{Uniform}(0,1)$ random variables, independent of each other and $N_t$.
The edge count process can be written as
$
\cE_t=\sum_{1\le i<j\le N_t} \ind_{\{\lambda\xi_{ij}\le \min(tU_i,tU_j)\}}.
$
Define $
\go_{ij} := \ind_{\xi_{ij} \le \theta (U_i \wedge U_j)}, 1\le i< j<\infty
$
so that
$
\cE_t = \sum_{1\le i<j\le N_t} \go_{ij}.
$
Define the functions 
\begin{equation*}
    h_\theta(u,v):=\pr(\xi\le \theta (u\wedge v))=1-e^{-\theta\min(u,v)}
\end{equation*}
and 
\begin{align*}
    H(\theta):=\E h_\theta(U_1,U_2)
&= \frac{2}{\theta^2}\bigl(1-\theta + \theta^2/2-e^{-\theta}\bigr)
= \frac\theta3(1+o(1)) \text{ when }\theta\ll1.
\end{align*}
Also, define the $\gs$-field $\cF_t$ generated by $\{N_t; U_i, 1\le i\le N_t\}$. Denote
\[
h_{ij}:=h_\theta(U_i,U_j) \text{ and }
\overline{\go}_{ij}:=\ind_{\xi_{ij}\le \theta (U_i\wedge U_j)}-h_{ij},
\]
Hence
\[
\cM_t:=\E(\cE_t\mid \cF_t)
= \sum_{1\le i<j\le N_t}h_{ij}.
\]
Using the fact that $\E N_t(N_t-1)=\lambda^2 t^2$, we have
\begin{align*}
    \mu_t=\E\cE_t
&=\frac12\lambda^2t^2H(t/\lambda)
= 
\begin{cases}
    \frac{1}{6}\lambda t^3\cdot (1+o(1))
&\text{ when }\theta\ll 1,\\
\frac12\lambda^2t^2 H(\theta) &\text{ when }\theta\approx 1.
\end{cases}
\end{align*}
In this case, instead of considering the centered-edge decomposition for $\cM_t$ conditioned on $N_t$ in the form of \eqref{eq:center_edge_intro}, we find it easier to work with the classical Hoeffding decomposition; nonetheless, it virtually accomplishes the same goal.
Since $\cM_t\mid N_t$ is a rank two U-statistic, and using Hoeffding's decomposition with $g_\theta(u):=\E h_\theta(U_1,u),$ we can write
\[
\cE_t-\mu_t = S_t+L_t+D_t+R_{t,1}+R_{t,2}
\]
where
\begin{align}
    S_t&:= \sum_{1\le i<j\le N_t}
\overline{\go}_{ij},\label{RGIV:cond_var2.0}\\
    L_t&:= (N_{t}-1) \sum_{i=1}^{N_t}\bigr(g_\theta(U_i)-H(\theta)\bigl),\label{RGIV:cond_var2.2}\\
    D_t&:=\lambda tH(\theta)\cdot (N_t-\lambda t),\label{RGIV:cond_var2.3}\\
    R_{t,1}&:=\sum_{1\le i<j\le N_t} \left(h_\theta(U_i,U_j) - g_\theta(U_i)- g_\theta(U_j) + H(\theta)\right),\label{RGIV:cond_var2.4}\\
    \text{and}\qquad R_{t,2}&:=\frac12H(\theta) ((N_t-\lambda t)^2-N_t).\label{RGIV:cond_var2.5}
\end{align}
The terms $R_{t,1}$ and $R_{t,2}$ are $O_P(t^2)$ as 
\begin{align*}
\E \var(R_{t,1}\mid N_t)
&= \E\binom{N_t}{2}\cdot \var(h_\theta(U_1,U_2) - g_\theta(U_1)- g_\theta(U_2))\\
&= \frac12\lambda^2t^2 \var(h_\theta(U_1,U_2) - g_\theta(U_1)- g_\theta(U_2))
\approx t^{4}
\end{align*}
and
\[
\E\abs{R_{t,2}}\le H(\theta)\lambda t\le t^2
\]
when $t=O(\lambda)$. Here, we used the fact that $h_\theta(u,v)=\theta(u\wedge v) +O(\theta^2)$ when $\theta\ll1$.

Now, for the centered Bernoulli sum $S_t$, the conditional variance can be written as
\[
\var(S_t\mid \cF_{t})=\sum_{1\le i<j\le N_t}(h_{2\theta}(U_i,U_j) - h_\theta(U_i,U_j)).
\]
Taking expectation 
 gives
\begin{align*}
\E \var(S_t\mid \cF_{t})
&= \E\binom{N_t}{2} \cdot \bigl(H(2\theta)-H(\theta)\bigr)\\
&= \E\frac{N_{t}(N_{t}-1)}{\lambda^{2}t^{2}}\cdot \frac{1}2\lambda^2 t^2\bigl(H(2\theta)-H(\theta)\bigr)
= \E \frac{N_{t}(N_{t}-1)}{\lambda^{2}t^{2}} \frac1{\theta}\bigl(H(2\theta)-H(\theta)\bigr)\cdot \frac{1}2\lambda t^3.
\end{align*}
Similarly,
\begin{align*}
    \E \var(L_t\mid N_t)
&= \var(g_\theta(U)) \E N_t(N_t-1)^2 
= (1+o(1))\cdot \begin{cases}
    \frac1{45}\lambda t^5 &\text{ when } t\ll \lambda,\\
    \lambda^3t^3\cdot \var(g_\theta(U)) &\text{ when } t\approx \lambda,
\end{cases}
\end{align*}
and
\begin{align*}
    \var(D_t) = H(\theta)^2 \lambda^3t^3
    = (1+o(1))\cdot \begin{cases}
    \frac19\lambda t^5 &\text{ when } t\ll \lambda,\\
    \lambda^3t^3\cdot H(\theta)^2 &\text{ when } t\approx \lambda.
\end{cases}
\end{align*}
For the limiting variance, what matters is whether $t\to\infty$ or not. Under our assumption, $\lambda$ is going to $\infty$, when $t=\Theta(1)$ both $S_t$ and $M_t$ terms contribute.

When $t\ll 1$, by the Stein-Chen method of Poisson approximation for the sum of independent Bernoulli random variables, conditional on $\cF_t$, we have
\[
\dtv(\cE_t,\poi(M_t))\le \max_{1\le i<j\le N_t}h_\theta(U_i,U_j)\le \theta,
\]
which is small since $t\ll 1\ll \lambda$.
Moreover, we have
\[
\dtv(\poi(M_t),\poi(\mu_t)))\le \E\abs{M_t - \mu_t} \le \norm{M_{t}-\mu_{t}}_{2}
\]
which is bounded by $O(t^{2}+t\sqrt{\lambda t^3})$. 

When $t=\Theta(1)$, the dominating terms are $S_{t}, L_t, D_t$ which jointly converge to independent Gaussian after proper scaling. When $1\ll t =O(\lambda)$, the dominating terms are~\eqref{RGIV:cond_var2.2} and~\eqref{RGIV:cond_var2.3}. The distributional limit follows from the Strong Law of Large Numbers, classical CLT for the sum of i.i.d.\ random variables, and the convergence of characteristic functions.

\subsection{Triangle count}

\begin{thm}[CLT for the triangle count]\label{thm:RGIV_T}
    Let $T_t$ be the number of triangles at time $t$.  The following convergences hold as $\lambda\to\infty$. 
    \begin{enumerate}[label=\upshape (\roman*)]
        \item\label{thm:RGIV_T_Poi}
        When $t\to c\in (0,\infty)$, we have 
        \[
        \dtv(T_t,\poi(c^6/(2\cdot 3^4)))\to 0.
        \]
       
        \item\label{thm:RGIV_T_N1} When $t\to\infty$ and $t=O(\lambda)$, we have 
        \[
        \dwas\left((T_t-\E T_t)/\sqrt{\var(T_t)},\N(0,1)\right)\to 0 
        \]
        where
        \[
        \var(T_t)\cong t^6(1+t^5/\lambda).
        \]
        \end{enumerate}
\end{thm}

Now, following similar steps to the number-of-edges analysis, we first analyze the triangle counts $T_t$ conditioned on the $\gs$-field $\cF_t$, which is determined by the number of vertices and their arrival times. 
We decompose $T_t$ conditioned on $\cF_{t}$ as
\[
T_t =\Delta_{t}+ T_{t,1}+T_{t,2} + T_{t,3},
\]
where 
\begin{align}
\Delta_{t}&:=\sum_{1\le i<j<k\le N_t}h_{ij}h_{ik}h_{jk}\label{RGIV_term:central0},\\
T_{t,1}&:=\sum_{i<j\le N_t}\left(\sum_{k\neq i,j}h_{ik}h_{jk}\right)\overline{\go}_{ij},\label{RGIV_term:central1}\\
T_{t,2}&:=\sum_{i<j\le N_t}h_{ij}\sum_{k\neq i,j} \overline{\go}_{ik}\overline{\go}_{jk},\label{RGIV_term:central2}\\
\text{and } \quad T_{t,3}&=\sum_{i<j<k\le N_t}\overline{\go}_{ij}\overline{\go}_{ik}\overline{\go}_{jk}\label{RGIV_term:central3}.
\end{align}
Recall that, $h_{ij}=h_\theta(U_i,U_j)$. Clearly, we have 
\[
\E T_t=\E(h_\theta(U_1,U_2)h_\theta(U_2,U_3)h_\theta(U_1,U_3))\cdot \E\binom{N_t}{3} \cong t^6
\]
when $t=O(\lambda)$. It is straightforward to verify that
\begin{align*}
    \E\var(T_{t,1}\mid \cF_t) 
    &\cong \E\sum_{1\le i<j\le N_t} h_{ij}\left(\sum_{k\neq i,j}h_{ik}h_{jk}\right)^2  \le \E(N_t^4 h_{12}^5) \cong t^9/\lambda,\\
    \E\var(T_{t,2}\mid \cF_t) 
    &\cong \E\sum_{1\le i<j\le N_t} h_{ij}^2 \left(\sum_{k\neq i,j}h_{ik}h_{jk}\right)  \le \E(N_t^3 h_{12}^4) \cong t^7/\lambda,\text{ and }\\
    \E\var(T_{t,3}\mid \cF_t)&\cong \E(N_t^3 h_{12}h_{23}h_{13})\cong t^6. 
\end{align*}
Now, conditionally on $N_t$, $\Delta_t$ is a rank three U-statistic with $\E\var(\Delta_t\mid N_t) \cong \E N_t^5 (t/\lambda)^6\cong t^{11}/\lambda.$ Finally, we have $\var(\E(T_t\mid N_t))\cong (t/\lambda)^6\var(N_t(N_t-1)(N_t-2))\cong t^{11}/\lambda.$
Thus, when $t\ll \lambda^{1/5}$, $T_{t,3}$ dominates, and  when $\lambda^{1/5}\ll t\lesssim \lambda$, $\gD_t$ dominates. Proof of the Poisson and Gaussian limit follows a similar approach to the edge-count case, using the classical Stein--Chen method for Poisson approximation, CLT for U-Statistics (see~\cite{UCLT}), and CLT for a polynomial of Poisson random variable.

\bibliographystyle{alpha}
\bibliography{pam} 
\end{document}